\documentclass[12pt,a4paper,reqno]{amsart}

\usepackage[T1]{fontenc}
\usepackage[utf8]{inputenc}
\usepackage{lmodern}
\usepackage{amsmath,amssymb,amsthm}
\usepackage[margin=1in]{geometry}
\usepackage[colorlinks=true,linkcolor=blue,citecolor=blue,urlcolor=blue]{hyperref}

\numberwithin{equation}{section}

\newtheorem{theorem}{Theorem}[section]
\newtheorem{corollary}[theorem]{Corollary}
\newtheorem{proposition}[theorem]{Proposition}
\newtheorem{lemma}[theorem]{Lemma}
\newtheorem{problem}[theorem]{Problem}

\theoremstyle{definition}
\newtheorem{remark}[theorem]{Remark}

\newcommand{\T}{\mathbb T}
\newcommand{\B}{\mathbb B}
\newcommand{\U}{\mathbb U}
\newcommand{\C}{\mathbb C}
\newcommand{\N}{\mathbb N}
\newcommand{\Hol}{\mathrm{Hol}}

\title[Isoperimetric-type inequalities for pluriharmonic functions]{Isoperimetric-type inequalities for pluriharmonic functions on the polydisc}

\author{Suman Das}
\address{Suman Das \vskip0.05cm Department of Mathematics, Shantou University, Shantou, Guangdong 515063, P. R. China.}
\email{suman@stu.edu.cn}

\author{Antti Rasila}
\address{Antti Rasila \vskip0.05cm Department of Mathematics with Computer Science, Guangdong Technion--Israel Institute of Technology, Shantou, Guangdong 515063, P. R. China; \vskip0.01cm
Department of Mathematics, Technion--Israel Institute of Technology, Haifa 3200003, Israel.}
\email{antti.rasila@gtiit.edu.cn; antti.rasila@iki.fi}

\author{Jian-Feng Zhu}
\address{Jian-Feng Zhu \vskip0.05cm Department of Mathematics, Shantou University, Shantou, Guangdong 515063, P. R. China.}
\email{flandy@stu.edu.cn}


\hypersetup{
 pdftitle={Isoperimetric-type inequalities for pluriharmonic functions on the polydisc},
 pdfauthor={Suman Das, Antti Rasila, and Jian-Feng Zhu},
 pdfsubject={Pluriharmonic Hardy and weighted Bergman inequalities},
 pdfkeywords={Pluriharmonic functions; Hardy spaces; Weighted Bergman spaces; Isoperimetric-type inequalities; Riesz-type inequalities}
}

\begin{document}

\begin{abstract}
	We prove isoperimetric-type inequalities for complex-valued pluriharmonic functions in the unit polydisc $\U^n\subset\C^n$. Denote by $h^p(\U^n)$ and $b^p_{\mathbf q}(\U^n)$, respectively, the pluriharmonic Hardy space and the pluriharmonic weighted Bergman space in $\U^n$, where $\mathbf q=(q_1,\ldots,q_n)>-\mathbf1$. For $m\in\N$, $m\geq2$, write $\mathbf{m-2}=(m-2,\ldots,m-2)$ and let
	\[
	d\mu_{\mathbf{m-2}}(z)
	=\frac{(m-1)^n}{\pi^n}
	\prod_{k=1}^n
	\left[(1-|z_k|^2)^{m-2}\,dx_kdy_k\right],
	\qquad z_k=x_k+iy_k.
	\]
	We prove that if $1<p_1,\ldots,p_m<\infty$ and $f_j\in h^{p_j}(\U^n)$, then
	\[
	\int_{\U^n}\prod_{j=1}^m|f_j(z)|^{p_j}\,
	d\mu_{\mathbf{m-2}}(z)
	\leq
	\prod_{j=1}^m
	\left[
	\frac{\sqrt2\cos\left(\frac{\pi}{2mp_j}\right)}
	{\sqrt{1-|\cos(\pi/p_j)|}}
	\right]^{p_j}
	\prod_{j=1}^m
	\|f_j\|_{h^{p_j}(\U^n)}^{p_j}.
	\]
	In particular,
	\[
	\|f\|_{b^{mp}_{\mathbf{m-2}}(\U^n)}
	\leq
	\frac{\sqrt2\cos\left(\frac{\pi}{2mp}\right)}
	{\sqrt{1-|\cos(\pi/p)|}}
	\|f\|_{h^p(\U^n)}.
	\]
	For $p>2$, we refine this in the form
	\[
	\|f\|_{b^{mp}_{\mathbf{m-2}}(\U^n)}
	\leq
	\left[
	\sqrt2\cos\left(\frac{\pi}{4m}\right)
	\right]^{2n/p}
	\|f\|_{h^p(\U^n)}.
	\]
	Consequently, the obtained constant in the diagonal inclusion tends to $1$ as $p\to\infty$, for fixed $m$ and $n$. When $m=2$ and $n=1$, the latter estimate coincides with best-known planar estimate. Explicit lower bounds at $p=2$, together with the dimension-free upper estimate, show that the optimal diagonal constants converge to $\sqrt2$ as $m\to\infty$, uniformly in the dimension. 
	
	Finally, we prove that if $f\in h^2(\U^n)$, then its restriction to the unit ball $\B_n\subset\C^n$ satisfies
	\[
	\|f\|_{h^{2n}(\B_n)}
	\leq
	\sqrt2\cos\left(\frac{\pi}{4n}\right)
	\|f\|_{h^2(\U^n)}.
	\]

\end{abstract}

\maketitle

\section{Introduction and Main Results}
\subsection{Notations}
Let $\U$ be the open unit disc and $\T$ be the unit circle. Denote by
\[
 \U^n=\U\times\cdots\times\U,
 \qquad
 \T^n=\T\times\cdots\times\T
 \qquad (n\text{ times})
\]
the unit polydisc and its distinguished boundary, respectively. We write $dm_n$ for normalized Haar measure on $\T^n$, that is,
\[
 dm_n(\omega)=\frac{d\theta_1\cdots d\theta_n}{(2\pi)^n},
 \qquad
 \omega=(e^{i\theta_1},\ldots,e^{i\theta_n}).
\]
For $0<p<\infty$, the holomorphic Hardy space $H^p(\U^n)$ consists of all holomorphic functions $F$ in $\U^n$ such that
\[
 \|F\|_{H^p(\U^n)}
 =\sup_{0<r<1}
 \left(
 \int_{\T^n}|F(r\omega)|^p\,dm_n(\omega)
 \right)^{1/p}<\infty.
\]

For $q>-1$, let
\[
 d\mu_q(z)=\frac{q+1}{\pi}(1-|z|^2)^q\,dx\,dy,
 \qquad z=x+iy\in\U,
\]
denote a normalized weighted measure on $\U$. We also consider the corresponding product measure on $\U^n$:
\[
 d\mu_{\mathbf q}(z)=\prod_{k=1}^n d\mu_{q_k}(z_k),
 \qquad z\in\U^n,
\]
where \(\mathbf{q}=(q_1,\ldots,q_n)>-\mathbf{1}\) is an \(n\)-multi-index; the inequality \(\mathbf{q}_1>\mathbf{q}_2\) between two \(n\)-multi-indices means that \(q_{1,k}>q_{2,k}\), \(k=1,\ldots,n\). We denote the \(n\)-multi-index \((m,\ldots,m)\) by \(\mathbf{m}\), and write \(\mathbf{m-2}=(m-2,\ldots,m-2)\). 
When a scalar \(a>-1\) is used as an index in \(d\mu_a\) on \(\mathbb U^n\), it denotes the constant multi-index \(\mathbf{a}=(a,\ldots,a)\). 
In particular, for a real number \(m>1\), \[ d\mu_{\mathbf{m-2}}(z) = \frac{(m-1)^n}{\pi^n} \prod_{k=1}^n (1-|z_k|^2)^{m-2}\,dx_kdy_k, \qquad z_k=x_k+iy_k . \]

The weighted Bergman spaces $A_{\mathbf q}^p(\U^n)$, $p>0$, $\mathbf q>-\mathbf1$, consist of all holomorphic functions $f$ in $\U^n$ such that
\[
 \|f\|_{A_{\mathbf q}^p(\U^n)}
 =\left(
 \int_{\U^n}|f(z)|^p\,d\mu_{\mathbf q}(z)
 \right)^{1/p}<\infty.
\]
We write $A^p(\U^n):=A_{\mathbf0}^p(\U^n)$ for the ordinary unweighted Bergman space. The classical theory of Hardy and Bergman spaces in the polydisc can be found in Rudin's monograph \cite{RudinPolydisc}.

Let $\B_n$ denote the unit ball in $\C^n$. We write \(d\sigma_n\) for normalized surface measure on \(\partial\B_n\) and \(d\nu_n\) for normalized Lebesgue measure on \(\B_n\), so that $\sigma_n(\partial\B_n)=\nu_n(\B_n)=1$. For \(0<p<\infty\), the Hardy space on $\B_n$ is the class of holomorphic functions $f$ such that
\[
 \|f\|_{H^p(\B_n)}
 =\sup_{0<r<1}
 \left(
 \int_{\partial\B_n}|f(r\zeta)|^p\,d\sigma_n(\zeta)
 \right)^{1/p} < \infty.
\]
We refer to the books of Rudin \cite{RudinBall} and Zhu \cite{ZhuSpaces} for the spaces of holomorphic functions in $\B_n$.

For finite-dimensional vector-valued holomorphic functions, the corresponding Hardy and Bergman norms are defined in the same way, with the Euclidean norm of the vector value in place of the scalar modulus. Throughout the paper, when $0<p<1$, all Hardy and Bergman quantities are understood as quasi-norms, and we retain the same notation.

\subsection{Pluriharmonic Functions}
A complex-valued function $f$ in $\U^n$ is called pluriharmonic if $f\in C^2(\U^n)$ and
\[
 \frac{\partial^2 f}{\partial z_j\,\partial \overline z_k}=0,
 \qquad 1\leq j,k\leq n.
\]
Since $\U^n$ is simply connected, this is equivalent to the representation
\[
 f=h+\overline g,
\]
where $h$ and $g$ are holomorphic in $\U^n$. The representation is unique after imposing $g(0)=0$, and we use this normalization throughout this paper.

The pluriharmonic Hardy space $h^p(\U^n)$ consists of the pluriharmonic functions $f$ for which
\[
 \|f\|_{h^p(\U^n)}
 =\sup_{0<r<1}
 \left(\int_{\T^n}|f(r\omega)|^p\,dm_n(\omega)\right)^{1/p}<\infty.
\]
We denote by $b_{\mathbf q}^p(\U^n)$ the weighted pluriharmonic Bergman space with norm
\[
 \|f\|_{b_{\mathbf q}^p(\U^n)}
 =\left(\int_{\U^n}|f(z)|^p\,d\mu_{\mathbf q}(z)\right)^{1/p}.
\]
When $0<p<1$, these expressions are understood as quasi-norms. The pluriharmonic Hardy space on the ball is defined analogously by
\[
 \|f\|_{h^p(\B_n)}
 =\sup_{0<r<1}
 \left(\int_{\partial\B_n}|f(r\zeta)|^p\,d\sigma_n(\zeta)\right)^{1/p}<\infty.
\]

\subsection{Background}
The classical isoperimetric inequality relates the length $L$ of a rectifiable Jordan curve to the area $A$ of the bounded domain that it encloses:
\[
A\leq\frac{L^2}{4\pi},
\]
with equality if and only if the curve is a circle.

The inequality has several important analytic formulations. Carleman \cite{Carleman} reduced it to an integral inequality for holomorphic functions in the unit disc. In our notation, Carleman's theorem reads
\[
\int_{\U}|F(z)|^2\,d\mu_0(z)
\leq
\left(\int_{\T}|F(\zeta)|\,dm_1(\zeta)\right)^2.
\]
Equivalently,
\[
 \|F\|_{A^2(\U)}\leq \|F\|_{H^1(\U)},\qquad F\in H^1(\U).
\]
This inequality is sharp, and equality is attained precisely by multiples of the square of the Szeg\H{o} kernel, that is, by functions of the form
\[
 F(z)=\frac{c}{(1-z\overline a)^2},
 \qquad c\in\C,\quad a\in\U.
\]
Carleman's theorem was subsequently extended and refined in several directions. Aronszajn \cite{Aronszajn} obtained a bilinear form of the inequality, while Strebel \cite[Theorem 19.9]{Strebel} established the inequality for all $p>0$. Mateljevi\'c and Pavlovi\'c \cite{MateljevicPavlovic} developed related isoperimetric inequalities on simply connected planar domains. Further developments may be found in the surveys of Gamelin and Khavinson \cite{GamelinKhavinson}, Osserman \cite{Osserman}, and Vukoti\'c \cite{Vukotic}.

A major advance was made by Burbea \cite{Burbea1987}, who proved a weighted family of inequalities. If \(m\in\N\), \(m\geq2\), $0<p<\infty$, and \(F\in H^p(\U)\), then
\begin{equation}\label{eq:Burbea-disk}
\int_{\U}|F(z)|^{mp}\,d\mu_{m-2}(z)
\leq
\left( \int_{\T} |F(\zeta)|^p\,dm_1(\zeta)\right)^m .
\end{equation}
Equivalently,
\[
H^p(\U)\subset A^{mp}_{m-2}(\U)
\]
with the inclusion operator having norm one. The classical Carleman--Strebel inclusion \(H^p\subset A^{2p}\) is recovered by taking \(m=2\). Burbea's theorem may be viewed as an analytic counterpart of a higher-order isoperimetric inequality, and has become one of the central results in this area.

Pavlovi\'c and Dostani\'c \cite{PavlovicDostanic} discovered a related multidimensional phenomenon. They proved the norm-one inclusion
\[
 H^2(\U^n)\subset H^{2n}(\B_n),
\]
or, equivalently,
\[
 \int_{\partial\B_n}|F(\zeta)|^{2n}\,d\sigma_n(\zeta)
 \leq
 \left(\int_{\T^n}|F(\omega)|^2\,dm_n(\omega)\right)^n,
 \qquad F\in H^2(\U^n).
\]
For $n\geq2$, restricting this inequality to functions of one coordinate recovers Burbea's weighted disc inequality.

Several multidimensional analogues have since been obtained. For example, Kalaj \cite{KalajPolydisk} established isoperimetric inequalities on generalized polydiscs. His result, in a special case, shows that for holomorphic functions \(F_1,F_2\) in \(\mathbb U^n\), \[ \int_{\mathbb U^n}|F_1(z)|^{p_1}|F_2(z)|^{p_2}\,d\mu_{\mathbf{0}}(z) \leq \int_{\mathbb T^n}|F_1(\omega)|^{p_1}\,dm_n(\omega) \int_{\mathbb T^n}|F_2(\omega)|^{p_2}\,dm_n(\omega). \]
The equality cases are described in terms of the powers of the Szeg\H{o} kernel. 

A sharp higher-dimensional analogue on the unit polydisc was proved by Markovi\'c \cite{Markovic}. If $m\in\N$, $m\geq2$, $p_j>0$, and $F_j\in H^{p_j}(\U^n)$, $j=1,\ldots,m$, then
\begin{equation}\label{eq:Markovic-polydisc}
 \int_{\U^n}
 \prod_{j=1}^m |F_j(z)|^{p_j}\,d\mu_{\mathbf{m-2}}(z)
 \leq
 \prod_{j=1}^m
 \int_{\T^n}|F_j(\omega)|^{p_j}\,dm_n(\omega).
\end{equation}
The equality cases are completely characterized. If none of the functions $F_j$ is identically zero, equality holds if and only if
\[
 F_j(z)=C_jK_n(z,\zeta)^{2/p_j},
\]
where $\zeta\in\U^n$ is common to all factors and $C_j\neq0$. If one factor is identically zero, both sides vanish and equality is automatic. Here
\[
 K_n(z,\zeta)=\prod_{k=1}^n\frac1{1-z_k\overline{\zeta_k}}
\]
is the Szeg\H{o} kernel of $H^2(\U^n)$. Taking all $F_j$ equal in \eqref{eq:Markovic-polydisc} gives the sharp inclusion
\[
 H^p(\U^n)\subset A^{mp}_{\mathbf{m-2}}(\U^n),
 \qquad
 \|F\|_{A^{mp}_{\mathbf{m-2}}}\leq \|F\|_{H^p}.
\]

Very recently, Chen, Gui, and Zhang \cite[Theorem~1.2]{ChenGuiZhang} obtained a real-parameter version of the sharp disc inclusion. In our notation, if $\alpha>1$, $p\geq1$, and $F\in H^p(\U)$, then
\[
 \|F\|_{A^{\alpha p}_{\alpha-2}(\U)}
 \leq \|F\|_{H^p(\U)}.
\]
Thus the contractive disc inclusion is available for every real weight parameter $\alpha>1$. We retain an integer parameter $m$, since it also counts the factors in the multilinear polydisc inequality.

The harmonic case has generated interest as well. Kalaj and Me\v{s}trovi\'c \cite{KalajMestrovic} obtained isoperimetric inequalities for planar harmonic functions. Kalaj and Bajrami \cite[Theorem~1.1]{KalajBajrami} proved that every real-valued harmonic function $f\in h^p(\U)$, $p>1$, satisfies
\[
 \|f\|_{b_0^{2p}(\U)}\leq R_{p,2}\|f\|_{h^p(\U)},
\]
where
\[
 R_{p,2}=
 \begin{cases}
 \dfrac{\cos(\pi/(4p))}{\cos(\pi/(2p))},&1<p\leq2,\\[2mm]
 \dfrac{\cos(\pi/(4p))}{\sin(\pi/(2p))},&p\geq2.
 \end{cases}
\]
The same estimate holds for complex-valued harmonic functions as well. Indeed, if $\varphi$ is the radial boundary function of $f$, positivity of the Poisson operator bounds $|f|$ by the Poisson extension of $|\varphi|$, whose $h^p$-norm is $\|\varphi\|_{L^p(\T)}=\|f\|_{h^p(\U)}$. Kalaj \cite{Kalaj2019} also obtained this estimate for integer exponents $p\geq2$ by means of sharp Riesz-type inequalities. For $p>2$, Melentijevi\'c \cite[Section~9]{Melentijevic} found, when $m=2$, an upper bound that tends to $1$ as $p\to\infty$. One purpose of this paper is to extend this high-$p$ phenomenon to every $m\geq2$ and to the polydisc.

Unlike the holomorphic inequalities, our pluriharmonic results require $p>1$. The strong endpoint inclusion fails even in the plane for every $m\geq2$. A counterexample and suitable substitutes under stronger hypotheses are given in Section~\ref{sec:endpoint}.

\subsection{Main Results}
For $k\in\N$, put
\[
 D_k=\sqrt2\cos\left(\frac{\pi}{4k}\right).
\]
For $m\in\N$, $m\geq2$, and $1<p<\infty$, put
\[
 R_{p,m}=\frac{\sqrt2\cos\left(\frac{\pi}{2mp}\right)}
 {\sqrt{1-|\cos(\pi/p)|}}.
\]
We then define
\begin{equation}\label{eq:Gamma-intro}
 \Gamma_{p,m,n}=
 \begin{cases}
 R_{p,m},&1<p\leq2,\\[1mm]
 \min\{R_{p,m},D_m^{2n/p}\},&p>2.
 \end{cases}
\end{equation}

Our first result is a pluriharmonic version of Markovi\'c's polydisc inequality \eqref{eq:Markovic-polydisc}.

\begin{theorem}\label{thm:intro-main}
Let $m\in\N$, $m\geq2$, $1<p_1,\ldots,p_m<\infty$, and let $f_j\in h^{p_j}(\U^n)$, $j=1,\ldots,m$. Then
\[
 \int_{\U^n}
 \prod_{j=1}^m |f_j(z)|^{p_j}\,d\mu_{\mathbf{m-2}}(z)
 \leq
 \prod_{j=1}^m\Gamma_{p_j,m,n}^{p_j}
 \prod_{j=1}^m
 \|f_j\|_{h^{p_j}(\U^n)}^{p_j}.
\]
\end{theorem}

In the diagonal case this gives
\begin{equation}\label{eq:intro-diag}
 \|f\|_{b^{mp}_{\mathbf{m-2}}(\U^n)}
 \leq
 \Gamma_{p,m,n}
 \|f\|_{h^p(\U^n)}.
\end{equation}
The constant $R_{p,m}$ is independent of the dimension $n$. When $m=2$ and $n=1$, $R_{p,2}$ is the Kalaj--Bajrami upper bound discussed above, so this planar diagonal specialization is already known. In particular, $R_{2,2}=D_2$, the upper bound obtained by Kalaj and Me\v{s}trovi\'c \cite[Theorem~1.2 and Remark~1.3]{KalajMestrovic}. We note that it is not known if these upper bounds are sharp.

The high-$p$ part of Theorem \ref{thm:intro-main} follows from a separate diagonal estimate. For fixed $m$ and $n$, it is substantially better than the above estimate when $p$ is sufficiently large.

\begin{theorem}\label{thm:intro-high-p}
Let $m\in\N$, $m\geq2$, $p>2$, and let $f\in h^p(\U^n)$. Then
\[
 \|f\|_{b^{mp}_{\mathbf{m-2}}(\U^n)}
 \leq
 D_m^{2n/p}\|f\|_{h^p(\U^n)}.
\]
\end{theorem}

For $1<p<\infty$, let $C_{p,m,n}$ denote the best constant in the diagonal inclusion
\[
 \|f\|_{b^{mp}_{\mathbf{m-2}}(\U^n)}
 \leq C_{p,m,n}\|f\|_{h^p(\U^n)}.
\]
For $p>2$, the constant function and the two upper estimates give
\[
 1\leq C_{p,m,n}\leq\Gamma_{p,m,n}.
\]
In particular,
\[
 \lim_{p\to\infty}C_{p,m,n}=1
\]
for fixed $m$ and $n$. When $m=2$ and $n=1$, the high-$p$ branch $D_2^{2/p}$ is precisely the bound obtained by Melentijevi\'c \cite[Section~9]{Melentijevic}.

Our next result is a pluriharmonic analogue of the Pavlovi\'c--Dostani\'c inclusion.

\begin{theorem}\label{thm:intro-PD}
If $f\in h^2(\U^n)$, then the restriction of $f$ to $\B_n$ satisfies
\[
 \|f\|_{h^{2n}(\B_n)}
 \leq
 D_n\|f\|_{h^2(\U^n)}.
\]
\end{theorem}

Our final main result concerns the optimal constants at $p=2$. The boundary-concentrating example used in the sharpness discussion leads to the explicit number
\begin{equation}\label{eq:L-definition}
 L_m=
 \left[
 \frac{(4m-2)!((m-1)!)^2}
 {2^m((2m-1)!)^2(2m-2)!}
 \right]^{1/(2m)}.
\end{equation}
It provides a lower bound for the optimal $p=2$ constants which, together with the explicit upper bound $D_m$, determines their limiting behavior uniformly in the dimension.

\begin{theorem}\label{thm:intro-Lm-asymptotic}
For every $m,n\in\N$, $m\geq2$,
\[
 L_m\leq C_{2,m,n}\leq D_m.
\]
Moreover,
\[
 L_m=\sqrt2\left(1-\frac{\log 2}{4m}+o(m^{-1})\right),
 \qquad m\to\infty,
\]
and
\[
 \lim_{m\to\infty}
 \sup_{n\in\N}|C_{2,m,n}-\sqrt2|=0.
\]
\end{theorem}

\subsection{Strategy of the Proofs}

Since $f=h+\overline g$, one might try to deduce the pluriharmonic inequalities directly from their holomorphic counterparts. A direct elementary reduction is ineffective because
\[
|h+\overline g|^p
\]
is not logarithmically plurisubharmonic in general.

Our approach is based on a square-function transfer principle. To a normalized decomposition $f=h+\overline g$, we associate the square function
\[
Q_f(z)=\bigl(|h(z)|^2+|g(z)|^2\bigr)^{1/2}.
\]
The function $Q_f$ is the Euclidean norm of the vector-valued holomorphic pair $(h,g)$. Sharp Riesz-type inequalities of Kalaj \cite{Kalaj2019}, extended to higher dimensions by Chen and Hamada \cite{ChenHamada}, compare $Q_f$ quantitatively with $f$ on the boundary. This is the primary technical tool of the paper.

The high-$p$ estimate is obtained differently. We iterate the explicit one-variable $h^2\to b^{2m}_{m-2}$ Poisson estimate over the coordinates and then use positivity and Jensen's inequality. This avoids the factor arising from the Riesz comparison, which grows linearly as $p\to\infty$.

To obtain lower bounds for the optimal $p=2$ constants, we construct a boundary-concentrating example for every $m\geq2$. A half-plane scaling and beta-integral calculation give the explicit constant in Theorem \ref{thm:intro-Lm-asymptotic}.

\subsection{Organization}
The paper is organized as follows. Section 2 presents the Riesz-type estimates
needed in the sequel and proves their weighted interior form. Section 3 gives a
self-contained proof of a vector-valued holomorphic polydisc inequality used for
the square functions. Section 4 proves the main polydisc inequality, the high-$p$
Poisson estimate, and its multilinear consequences. Section 5
proves the ball--polydisc inclusion and a corresponding volume corollary.
Section 6 records partial results for the case $0<p\leq1$ under square-function
or split Hardy hypotheses. In Section 7 we prove explicit lower bounds for the
optimal $p=2$ constants, derive their limiting consequences, and conclude with
a discussion of the remaining sharpness problem.

\section{Riesz-Type Estimates on the Polydisc}

We shall use the constants
\[
 A_p=\frac{1}{\sqrt{1-|\cos(\pi/p)|}},
 \qquad
 B_p=\sqrt2\max\left\{
 \sin\left(\frac{\pi}{2p}\right),
 \cos\left(\frac{\pi}{2p}\right)
 \right\},
 \qquad 1<p<\infty.
\]
For $p\geq2$,
\[
 B_p=\sqrt2\cos\left(\frac{\pi}{2p}\right).
\]
Consequently, the constants introduced in Section~1 satisfy
\[
 R_{p,m}=A_pB_{mp},
 \qquad
 D_m=B_{2m}.
\]

Chen and Hamada \cite[Theorem~2.1]{ChenHamada} extended the sharp disc inequalities of Kalaj \cite{Kalaj2019} to bounded symmetric domains. We quote the polydisc form required in Sections 2--4. The unit-ball form of the same theorem, with the same constants and sign conditions, is used in Section 5.

\begin{theorem}\label{thm:CH}
Let $1<p<\infty$, and suppose that
\[
 f=h+\overline g\in h^p(\U^n),
\]
where $h$ and $g$ are holomorphic in $\U^n$.
If
\[
 \operatorname{Re}(g(0)h(0))\geq0,
\]
then
\begin{equation}\label{eq:CH-A}
 \lim_{r\to1^-}
 \left(
 \int_{\T^n}(|h(r\omega)|^2+|g(r\omega)|^2)^{p/2}\,dm_n(\omega)
 \right)^{1/p}
 \leq A_p\|f\|_{h^p(\U^n)}.
\end{equation}
If
\[
 \operatorname{Re}(g(0)h(0))\leq0,
\]
then
\begin{equation}\label{eq:CH-B}
 \|f\|_{h^p(\U^n)}
 \leq
 B_p
 \lim_{r\to1^-}
 \left(
 \int_{\T^n}(|h(r\omega)|^2+|g(r\omega)|^2)^{p/2}\,dm_n(\omega)
 \right)^{1/p}.
\end{equation}
\end{theorem}

In our applications, the normalization $g(0)=0$ makes both sign conditions available. To use the boundary comparisons inside weighted Bergman integrals, we apply them to anisotropic dilations on the internal tori of the polydisc and then integrate over the radii. This gives the following two-sided weighted form.

\begin{lemma}\label{lem:weighted-B}
Let $1<q<\infty$ and let $\boldsymbol\alpha=(\alpha_1,\ldots,\alpha_n)>-\mathbf1$. Suppose $f=h+\overline g$ is pluriharmonic in $\U^n$ with $g(0)=0$. Define
\[
 Q_f(z)=\bigl(|h(z)|^2+|g(z)|^2\bigr)^{1/2}.
\]
Then
\begin{equation}\label{eq:weighted-A}
 A_q^{-q}\int_{\U^n}Q_f(z)^q\,d\mu_{\boldsymbol\alpha}(z)
 \leq
 \int_{\U^n}|f(z)|^q\,d\mu_{\boldsymbol\alpha}(z),
\end{equation}
and
\begin{equation}\label{eq:weighted-B}
 \int_{\U^n}|f(z)|^q\,d\mu_{\boldsymbol\alpha}(z)
 \leq
 B_q^q
 \int_{\U^n}Q_f(z)^q\,d\mu_{\boldsymbol\alpha}(z),
\end{equation}
whenever the respective right-hand sides are finite.
\end{lemma}

\begin{proof}

Fix $0<\rho<1$. For $\boldsymbol r=(r_1,\ldots,r_n)$ with $0<r_k<\rho$, define
\[
 f_{\boldsymbol r}(\zeta)=f(r_1\zeta_1,\ldots,r_n\zeta_n),
 \qquad \zeta\in\U^n.
\]
Then
\[
 f_{\boldsymbol r}=h_{\boldsymbol r}+\overline{g_{\boldsymbol r}},
\]
where
\[
 h_{\boldsymbol r}(\zeta)=h(r_1\zeta_1,\ldots,r_n\zeta_n),
 \qquad
 g_{\boldsymbol r}(\zeta)=g(r_1\zeta_1,\ldots,r_n\zeta_n).
\]
Since $g(0)=0$, we also have $g_{\boldsymbol r}(0)=0$. Therefore,
\[
 \operatorname{Re}(g_{\boldsymbol r}(0)h_{\boldsymbol r}(0))=0,
\]
	so both sign conditions in Theorem \ref{thm:CH} are satisfied.

	Moreover, because each $r_k<\rho<1$, the functions $h_{\boldsymbol r}$ and $g_{\boldsymbol r}$ are holomorphic in a neighbourhood of the closed polydisc. The radial limits in \eqref{eq:CH-A} and \eqref{eq:CH-B} are therefore the corresponding boundary integrals on $\T^n$. Since the real and imaginary parts of $f_{\boldsymbol r}$ are separately harmonic and $w\mapsto|w|^q$ is convex for $q>1$, the function $|f_{\boldsymbol r}|^q$ is separately subharmonic. Successive applications of the one-variable submean inequality show that its torus means are nondecreasing in each radius. Since $f_{\boldsymbol r}$ is continuous on $\overline{\U^n}$, it follows that
\[
 \|f_{\boldsymbol r}\|_{h^q(\U^n)}^q
 =\int_{\T^n}|f_{\boldsymbol r}(\omega)|^q\,dm_n(\omega).
\]
Applying both parts of Theorem \ref{thm:CH} to $f_{\boldsymbol r}$ now gives the fixed-radii estimates
\[
 A_q^{-q}\int_{\T^n}
 \bigl(|h_{\boldsymbol r}(\omega)|^2+|g_{\boldsymbol r}(\omega)|^2\bigr)^{q/2}
 \,dm_n(\omega)
 \leq
 \int_{\T^n}|f_{\boldsymbol r}(\omega)|^q\,dm_n(\omega)
\]
and
\[
 \int_{\T^n}
 |f(r_1\omega_1,\ldots,r_n\omega_n)|^q\,dm_n(\omega)
 \leq
 B_q^q
 \int_{\T^n}
 Q_f(r_1\omega_1,\ldots,r_n\omega_n)^q\,dm_n(\omega).
\]

We now integrate the two fixed-radii estimates over the radii. In one variable, if $z_k=r_ke^{i\theta_k}$ and $dm_1(e^{i\theta_k})=d\theta_k/(2\pi)$, then
\[
 d\mu_{\alpha_k}(z_k)
 =2(\alpha_k+1)(1-r_k^2)^{\alpha_k}r_k\,dr_k\,dm_1(e^{i\theta_k}).
\]
Therefore, on the polydisc,
\[
 d\mu_{\boldsymbol\alpha}(z)
 =\prod_{k=1}^n
 2(\alpha_k+1)(1-r_k^2)^{\alpha_k}r_k\,dr_k\,dm_1(e^{i\theta_k}).
\]
Multiplying the fixed-radii estimates by
\[
 \prod_{k=1}^n 2(\alpha_k+1)(1-r_k^2)^{\alpha_k}r_k\,dr_k
\]
and integrating over $0<r_k<\rho$, $1\leq k\leq n$, gives
\[
 A_q^{-q}\int_{\rho\U^n}Q_f(z)^q\,d\mu_{\boldsymbol\alpha}(z)
 \leq
 \int_{\rho\U^n}|f(z)|^q\,d\mu_{\boldsymbol\alpha}(z)
 \leq
 B_q^q
 \int_{\rho\U^n}Q_f(z)^q\,d\mu_{\boldsymbol\alpha}(z).
\]
Finally, the domains $\rho\U^n$ increase to $\U^n$ as $\rho\to1^-$. Since the integrands are non-negative, monotone convergence yields \eqref{eq:weighted-A} and \eqref{eq:weighted-B}.
\end{proof}

Writing $H_f=(h,g)$, Lemma \ref{lem:weighted-B} gives the norm equivalence
\[
 A_q^{-1}\|H_f\|_{A^q_{\boldsymbol\alpha}(\U^n;\C^2)}
 \leq \|f\|_{b^q_{\boldsymbol\alpha}(\U^n)}
 \leq B_q\|H_f\|_{A^q_{\boldsymbol\alpha}(\U^n;\C^2)}.
\]
Thus $f\in b^q_{\boldsymbol\alpha}(\U^n)$ if and only if the pair $H_f$ belongs to $A^q_{\boldsymbol\alpha}(\U^n;\C^2)$.

\section{A Vector-Valued Polydisc Inequality}

In this section, we prove the Banach space-valued form of the polydisc inequality needed for the holomorphic pair $(h,g)$. It is the natural Banach space extension of Markovi\'c's $\C^N$-valued theorem \cite[Theorem~2.7]{Markovic}, which was stated without a proof. For the convenience of the reader, and for completeness, we include a self-contained argument. It uses only one-variable logarithmic subharmonicity, Burbea's disc inequality, and induction on the number of variables.

A non-negative function \(u\) in a domain \(D\subset\C\) is called logarithmically subharmonic if either \(u\equiv0\), or \(\log u\) is subharmonic in \(D\). The induction below requires logarithmic subharmonicity of the norm of a Banach space-valued holomorphic function. We record the elementary fact needed for this purpose.

\begin{lemma}\label{lem:banach-logsubharmonic}
	Let \(X\) be a complex Banach space and let \(F:D\to X\) be holomorphic.  Then \(z\mapsto \|F(z)\|\) is logarithmically subharmonic in \(D\).  Consequently, for every \(p>0\), the function \(z\mapsto \|F(z)\|^p\) is logarithmically subharmonic.
\end{lemma}

\begin{proof}
	If \(F\equiv0\), there is nothing to prove.  Assume that \(F\not\equiv0\).  Fix a closed disc \(\overline{D(a,r)}\subset D\).  If \(F(a)=0\), then the submean inequality for \(\log\|F\|\) at \(a\) is immediate, since the left-hand side is \(-\infty\).  Suppose \(F(a)\neq0\).  By the Hahn--Banach theorem, there exists a continuous linear functional \(\Lambda\in X^*\) with \(\|\Lambda\|=1\) such that
	\[
	\Lambda(F(a))=\|F(a)\|.
	\]
	The scalar function \(\Lambda\circ F\) is holomorphic.  Therefore \(\log|\Lambda(F(z))|\) is subharmonic, and hence
	\[
	\log\|F(a)\|
	=
	\log|\Lambda(F(a))|
	\leq
	\frac1{2\pi}\int_0^{2\pi}
	\log|\Lambda(F(a+re^{it}))|\,dt.
	\]
	Since \(|\Lambda(F(w))|\leq\|F(w)\|\), we obtain
	\[
	\log\|F(a)\|
	\leq
	\frac1{2\pi}\int_0^{2\pi}
	\log\|F(a+re^{it})\|\,dt.
	\]
	Thus \(\log\|F\|\) satisfies the submean inequality. It is also locally integrable. Indeed, choose $\Lambda_0\in X^*$ such that $\Lambda_0\circ F\not\equiv0$. Then $\log\|F\|\geq\log|\Lambda_0\circ F|-\log\|\Lambda_0\|$, while $\log\|F\|$ is locally bounded above. Since it is upper semicontinuous, $\log\|F\|$ is subharmonic. Multiplying by \(p>0\) gives the assertion for \(\|F\|^p\).
\end{proof}

We next extend Burbea's one-variable estimate from moduli of holomorphic functions to continuous logarithmically subharmonic functions. This provides the base case of the polydisc induction.

\begin{lemma}\label{lem:one-variable-log}
	Let \(m\in\N\), \(m\geq2\).  Suppose \(u_1,\ldots,u_m\) are continuous non-negative logarithmically subharmonic functions in a neighbourhood of \(\overline\U\).  Then
	\begin{equation}\label{eq:one-variable-log}
		\int_{\U}\prod_{j=1}^m u_j(z)\,d\mu_{m-2}(z)
		\leq
		\prod_{j=1}^m\int_{\T}u_j(\zeta)\,dm_1(\zeta).
	\end{equation}
\end{lemma}

\begin{proof}
	If one of the functions \(u_j\) is identically zero, the assertion is trivial.  Therefore, assume that none of them is identically zero.  Fix \(\varepsilon>0\).  Since each \(u_j\) is continuous and logarithmically subharmonic in a neighbourhood of \(\overline\U\), the boundary function \(\log(u_j+\varepsilon)\) is continuous and integrable on \(\T\).  Let \(O_{j,\varepsilon}\) be the outer function in \(\U\) with the boundary modulus
	\[
	|O_{j,\varepsilon}(\zeta)|=u_j(\zeta)+\varepsilon,
	\qquad \zeta\in\T.
	\]
	Explicitly,
	\[
	O_{j,\varepsilon}(z)=
	\exp\left\{
	\int_{\T}\frac{\zeta+z}{\zeta-z}\log(u_j(\zeta)+\varepsilon)\,dm_1(\zeta)
	\right\}.
	\]
	Since $\log(u_j+\varepsilon)$ is bounded on $\T$, the function $O_{j,\varepsilon}$ belongs to $H^\infty(\U)$.
	Since \(\log u_j\) is subharmonic and \(\log u_j\leq \log(u_j+\varepsilon)\) on \(\T\), the maximum principle for subharmonic functions gives
	\[
	\log u_j(z)
	\leq
	\int_{\T}P_z(\zeta)\log(u_j(\zeta)+\varepsilon)\,dm_1(\zeta)
	=
	\log|O_{j,\varepsilon}(z)|,
	\]
	where \(P_z\) is the Poisson kernel.  Hence
	\begin{equation}\label{eq:outer-majorant}
		u_j(z)\leq |O_{j,\varepsilon}(z)|,
		\qquad z\in\U.
	\end{equation}

	We now use Burbea's one-variable weighted inequality.  For a holomorphic function \(G\in H^1(\U)\), inequality \eqref{eq:Burbea-disk} with \(p=1\) gives
	\[
	\int_{\U}|G(z)|^m\,d\mu_{m-2}(z)
	\leq
	\left(\int_{\T}|G(\zeta)|\,dm_1(\zeta)\right)^m.
	\]
	Applying H\"older's inequality to the product of the functions \(O_{j,\varepsilon}\), we get
	\[
	\int_{\U}\prod_{j=1}^m |O_{j,\varepsilon}(z)|\,d\mu_{m-2}(z)
	\leq
	\prod_{j=1}^m
	\left(\int_{\U}|O_{j,\varepsilon}(z)|^m\,d\mu_{m-2}(z)\right)^{1/m}.
	\]
	By Burbea's inequality, each factor is bounded by
	\[
	\int_{\T}|O_{j,\varepsilon}(\zeta)|\,dm_1(\zeta)
	=
	\int_{\T}\bigl(u_j(\zeta)+\varepsilon\bigr)\,dm_1(\zeta).
	\]
	Together with \eqref{eq:outer-majorant}, this yields
	\[
	\int_{\U}\prod_{j=1}^m u_j(z)\,d\mu_{m-2}(z)
	\leq
	\prod_{j=1}^m
	\int_{\T}\bigl(u_j(\zeta)+\varepsilon\bigr)\,dm_1(\zeta).
	\]
	Letting \(\varepsilon\to0^+\) proves \eqref{eq:one-variable-log}.
\end{proof}

The next proposition is the Banach space-valued polydisc inequality needed for the square functions. Its constant is inherited from the scalar holomorphic inequality.

\begin{proposition}\label{prop:vector-polydisc}
	Let \(m\in\N\), \(m\geq2\), \(1\leq p_1,
	\ldots,p_m<\infty\), and let \(X_1,
		\ldots,X_m\) be complex Banach spaces. Suppose that \(F_j:\U^n\to X_j\) is holomorphic and
	\[
	\sup_{0<r<1}\int_{\T^n}\|F_j(r\omega)\|_{X_j}^{p_j}\,dm_n(\omega)<\infty,
	\qquad j=1,
	\ldots,m.
	\]
	Then
	\begin{equation}\label{eq:banach-polydisc}
		\int_{\U^n}\prod_{j=1}^m\|F_j(z)\|_{X_j}^{p_j}\,d\mu_{\mathbf{m-2}}(z)
		\leq
		\prod_{j=1}^m
		\sup_{0<r<1}\int_{\T^n}\|F_j(r\omega)\|_{X_j}^{p_j}\,dm_n(\omega).
	\end{equation}
	If, for each $j$, there is a radial boundary function $F_j^*\in L^{p_j}(\T^n;X_j)$ such that $F_j(r\,\cdot)\to F_j^*$ in $L^{p_j}$ as $r\to1^-$, then the right-hand side is
		\[
		\prod_{j=1}^m
		\int_{\T^n}\|F_j^*(\omega)\|_{X_j}^{p_j}\,dm_n(\omega).
	\]
\end{proposition}
\begin{proof}
	We first prove the assertion under the additional assumption that each \(F_j\) is holomorphic in a neighbourhood of the closed polydisc \(\overline{\U^n}\).  We argue by induction on \(n\).

	For \(n=1\), Lemma \ref{lem:banach-logsubharmonic} shows that
	\[
	u_j(z)=\|F_j(z)\|_{X_j}^{p_j}
	\]
	is logarithmically subharmonic in a neighbourhood of \(\overline\U\). Thus, Lemma \ref{lem:one-variable-log} gives
	\[
	\int_{\U}\prod_{j=1}^m\|F_j(z)\|_{X_j}^{p_j}\,d\mu_{m-2}(z)
	\leq
	\prod_{j=1}^m
	\int_{\T}\|F_j(\zeta)\|_{X_j}^{p_j}\,dm_1(\zeta).
	\]
	This proves the case \(n=1\).

	Assume now that the result has been proved in \(n-1\) variables for Banach space-valued holomorphic functions.  Write \(z=(z',z_n)\), where \(z'\in\U^{n-1}\).  For fixed \(z'\), apply the one-variable case to the functions
	\[
	z_n\mapsto F_j(z',z_n).
	\]
	This gives
	\begin{align*}
		\int_{\U}\prod_{j=1}^m\|F_j(z',z_n)\|_{X_j}^{p_j}\,d\mu_{m-2}(z_n)
		\leq
		\prod_{j=1}^m
		\int_{\T}\|F_j(z',\eta)\|_{X_j}^{p_j}\,dm_1(\eta).
	\end{align*}
	For each \(j\), define the Bochner space
	\[
	Y_j=L^{p_j}(\T;X_j)
	\]
	and define a function \(G_j:\U^{n-1}\to Y_j\) by
	\[
	G_j(z')(\eta)=F_j(z',\eta),
	\qquad \eta\in\T.
	\]
		Because \(F_j\) is holomorphic in a neighbourhood of \(\overline{\U^n}\), the difference quotients in each variable of \(z'\) converge uniformly in \(\eta\in\T\). Hence they converge in \(L^{p_j}(\T;X_j)\). The Banach-valued Cauchy formula then shows that $G_j$ is jointly holomorphic in a neighbourhood of $\overline{\U^{n-1}}$. Moreover,
	\[
	\|G_j(z')\|_{Y_j}^{p_j}
	=
	\int_{\T}\|F_j(z',\eta)\|_{X_j}^{p_j}\,dm_1(\eta).
	\]
	Integrating the preceding one-variable estimate over \(z'\in\U^{n-1}\) and applying the induction hypothesis to the Banach-valued maps \(G_j\), we obtain
	\begin{align*}
		\int_{\U^n}\prod_{j=1}^m\|F_j(z)\|_{X_j}^{p_j}\,d\mu_{\mathbf{m-2}}(z)
		&\leq
		\int_{\U^{n-1}}\prod_{j=1}^m\|G_j(z')\|_{Y_j}^{p_j}\,d\mu_{\mathbf{m-2}}(z') \\
		&\leq
		\prod_{j=1}^m
		\int_{\T^{n-1}}\|G_j(\omega')\|_{Y_j}^{p_j}\,dm_{n-1}(\omega') \\
		&=
		\prod_{j=1}^m
		\int_{\T^{n-1}}\int_{\T}\|F_j(\omega',\eta)\|_{X_j}^{p_j}\,dm_1(\eta)\,dm_{n-1}(\omega') \\
		&=
		\prod_{j=1}^m
		\int_{\T^n}\|F_j(\omega)\|_{X_j}^{p_j}\,dm_n(\omega).
	\end{align*}
	This proves the case of the closed polydisc.

	For the general case, fix \(0<\rho<1\) and apply the proved estimate to the dilated maps
	\[
	F_{j,\rho}(z)=F_j(\rho z).
	\]
	They are holomorphic in a neighbourhood of \(\overline{\U^n}\).  Thus
	\[
	\int_{\U^n}\prod_{j=1}^m\|F_j(\rho z)\|_{X_j}^{p_j}\,d\mu_{\mathbf{m-2}}(z)
	\leq
	\prod_{j=1}^m
	\int_{\T^n}\|F_j(\rho\omega)\|_{X_j}^{p_j}\,dm_n(\omega).
	\]
		Let \(\rho\to1^-\). The integrand on the left converges pointwise to \(\prod_j\|F_j(z)\|^{p_j}\), and Fatou's lemma gives the left-hand side of \eqref{eq:banach-polydisc}. The right-hand side is bounded by the product of the radial Hardy quantities appearing there. This proves the proposition. By Lemma \ref{lem:banach-logsubharmonic}, each function $\|F_j\|_{X_j}^{p_j}$ is separately subharmonic. Successive applications of the one-variable submean inequality therefore show that the corresponding radial integrals are nondecreasing. If the stated $L^{p_j}$ boundary convergence holds, convergence of the $L^{p_j}$ norms identifies their limits with the boundary integrals, giving the final assertion.
\end{proof}

Applying the preceding proposition to repeated copies of a holomorphic pair gives the square-function estimate used below. We record this diagonal consequence separately.

\begin{corollary}\label{cor:square-Markovic}
	Let \(m\in\N\), \(m\geq2\), \(1\leq p<\infty\), and let \(h,g\) be holomorphic in \(\U^n\).  Define
	\[
	H=(h,g),
	\qquad
	Q(z)=\|H(z)\|=\left(|h(z)|^2+|g(z)|^2\right)^{1/2}.
	\]
	Assume that
	\[
	\|H\|_{H^p(\U^n;\C^2)}^p
	:=
	\sup_{0<r<1}\int_{\T^n}Q(r\omega)^p\,dm_n(\omega)<\infty.
	\]
	Then
	\begin{equation}\label{eq:square-cor}
		\int_{\U^n}Q(z)^{mp}\,d\mu_{\mathbf{m-2}}(z)
		\leq
		\|H\|_{H^p(\U^n;\C^2)}^{mp}.
	\end{equation}
		The radial boundary function $H^*\in L^p(\T^n;\C^2)$ exists. Writing $Q^*(\omega)=\|H^*(\omega)\|$, inequality \eqref{eq:square-cor} may equivalently be written as
		\[
		\int_{\U^n}Q(z)^{mp}\,d\mu_{\mathbf{m-2}}(z)
		\leq
		\left(\int_{\T^n}Q^*(\omega)^p\,dm_n(\omega)\right)^m.
	\]
\end{corollary}

\begin{proof}
	Since $H\in H^p(\U^n;\C^2)$, its two components belong to the scalar Hardy space $H^p(\U^n)$. The scalar Hardy boundary theorem on the polydisc \cite{RudinPolydisc}, applied componentwise, gives $H^*\in L^p(\T^n;\C^2)$ and
	\[
	 H(r\,\cdot)\longrightarrow H^*
	 \quad\text{in }L^p(\T^n;\C^2)
	 \quad\text{as }r\to1^-.
	\]
	Apply Proposition \ref{prop:vector-polydisc} with \(X_1=\cdots=X_m=\C^2\),
	\[
	F_1=\cdots=F_m=H,
	\qquad
	p_1=\cdots=p_m=p.
	\]
	Since \(\|H(z)\|=Q(z)\), the two forms of that proposition give the asserted inequalities.
\end{proof}

\subsection{The Square-Function Transfer}

The preceding corollary supplies the holomorphic part of the argument. Its importance for pluriharmonic functions is that it applies to the square function associated with the holomorphic decomposition. If
\[
 f=h+\overline g,
\]
then Markovi\'c's holomorphic inequality cannot be applied directly to $|f|$. However, if we introduce the vector-valued holomorphic function
\[
 H_f=(h,g),
\]
then
\[
 Q_f=\|H_f\|=(|h|^2+|g|^2)^{1/2}.
\]
Hence the preceding vector-valued holomorphic inequality applies to $H_f$, and then Lemma \ref{lem:weighted-B} gives a quantitative way to pass from $f$ to $Q_f$ and back.

\section{Weighted Inequalities on the Polydisc}

We first establish the dimension-free Riesz-transfer part of the main polydisc theorem. This is the estimate obtained by passing through the holomorphic square functions.

\begin{proposition}\label{prop:Riesz-transfer}
Let $m\in\N$, $m\geq2$, $1<p_1,\ldots,p_m<\infty$, and let $f_j\in h^{p_j}(\U^n)$, $j=1,\ldots,m$. Then
\[
 \int_{\U^n}\prod_{j=1}^m|f_j(z)|^{p_j}\,d\mu_{\mathbf{m-2}}(z)
 \leq
 \prod_{j=1}^mR_{p_j,m}^{p_j}
 \prod_{j=1}^m\|f_j\|_{h^{p_j}(\U^n)}^{p_j}.
\]
\end{proposition}

\begin{proof}
For each $j$, write
\[
 f_j=h_j+\overline{g_j},\qquad g_j(0)=0,
\]
where $h_j,g_j$ are holomorphic in $\U^n$. Put
\[
 H_j=(h_j,g_j),\qquad
 Q_j(z)=\|H_j(z)\|=\bigl(|h_j(z)|^2+|g_j(z)|^2\bigr)^{1/2}.
\]

By Lemma \ref{lem:banach-logsubharmonic}, $Q_j^{p_j}=\|H_j\|^{p_j}$ is separately subharmonic. Its product-torus radial means are therefore nondecreasing, and the forward Chen--Hamada estimate \eqref{eq:CH-A} gives
\[
 \sup_{0<r<1}\int_{\T^n}Q_j(r\omega)^{p_j}\,dm_n(\omega)
 \leq
 A_{p_j}^{p_j}\|f_j\|_{h^{p_j}(\U^n)}^{p_j}.
\]
Thus \(H_j\in H^{p_j}(\U^n;\C^2)\). Corollary \ref{cor:square-Markovic} then gives
\[
 \int_{\U^n}Q_j^{mp_j}\,d\mu_{\mathbf{m-2}}<\infty.
\]
Consequently, the weighted Riesz estimate in Lemma \ref{lem:weighted-B} can be applied with $q=mp_j$ and $\boldsymbol\alpha=\mathbf{m-2}$.

H\"older's inequality gives
\begin{equation}\label{eq:Holder-step}
 \int_{\U^n}\prod_{j=1}^m |f_j|^{p_j}\,d\mu_{\mathbf{m-2}}
 \leq
 \prod_{j=1}^m
 \left(
 \int_{\U^n}|f_j|^{mp_j}\,d\mu_{\mathbf{m-2}}
 \right)^{1/m}.
\end{equation}
Moreover, for each $j$, Lemma \ref{lem:weighted-B}, Corollary \ref{cor:square-Markovic}, and the preceding supremum estimate give
\[
\begin{aligned}
 \left(\int_{\U^n}|f_j|^{mp_j}\,d\mu_{\mathbf{m-2}}\right)^{1/m}
 &\leq
 B_{mp_j}^{p_j}
 \left(\int_{\U^n}Q_j^{mp_j}\,d\mu_{\mathbf{m-2}}\right)^{1/m} \\
 &\leq
 B_{mp_j}^{p_j}
 \sup_{0<r<1}\int_{\T^n}Q_j(r\omega)^{p_j}\,dm_n(\omega) \\
 &\leq
 R_{p_j,m}^{p_j}\|f_j\|_{h^{p_j}(\U^n)}^{p_j}.
\end{aligned}
\]
Combining this estimate with \eqref{eq:Holder-step} and using $R_{p_j,m}=A_{p_j}B_{mp_j}$, we get the desired conclusion.
\end{proof}

Taking equal functions and equal exponents in the preceding proposition gives the corresponding Hardy-to-weighted-Bergman inclusion. We record both its integral and norm forms.

\begin{corollary}\label{cor:diagonal-Markovic}
Let $m\in\N$, $m\geq2$, $1<p<\infty$, and let $f\in h^p(\U^n)$. Then
\begin{equation}\label{eq:diag-int}
 \int_{\U^n}|f(z)|^{mp}\,d\mu_{\mathbf{m-2}}(z)
 \leq
 R_{p,m}^{mp}
 \|f\|_{h^p(\U^n)}^{mp}.
\end{equation}
Equivalently,
\begin{equation}\label{eq:diag-norm}
 \|f\|_{b^{mp}_{\mathbf{m-2}}(\U^n)}
 \leq
 R_{p,m}\,
 \|f\|_{h^p(\U^n)}.
\end{equation}
\end{corollary}

\begin{proof}
Take $f_1=\cdots=f_m=f$ and $p_1=\cdots=p_m=p$ in Proposition \ref{prop:Riesz-transfer}. This gives \eqref{eq:diag-int}; taking $mp$-th roots gives \eqref{eq:diag-norm}.
\end{proof}

\begin{remark}
When $g\equiv0$, the sharp holomorphic inclusion is stronger and has norm one by Markovi\'c's theorem. The constants in Corollary \ref{cor:diagonal-Markovic} are intended for the full pluriharmonic class and arise from the Riesz comparison between $f$ and its square function.
\end{remark}

\begin{remark}
The elementary estimate
\[
 |h+\overline g|\leq \sqrt2(|h|^2+|g|^2)^{1/2}
\]
would replace $B_{mp}$ by $\sqrt2$ and give the constant $\sqrt2 A_p$. Since $B_{mp}<\sqrt2$, the Riesz-transfer constant $R_{p,m}=A_pB_{mp}$ is strictly smaller.
\end{remark}

\begin{remark}
If $m=2$ and $p=N\geq2$ is an integer, then $R_{N,2}$ agrees with the upper-bound constant obtained by Kalaj \cite[Theorem~2.11]{Kalaj2019}. In particular, $R_{2,2}=D_2\approx1.306563$, which is also the upper bound obtained by Kalaj and Me\v{s}trovi\'c \cite[Theorem~1.2 and Remark~1.3]{KalajMestrovic}. When $n=1$ and $N>2$, Theorem \ref{thm:intro-high-p} gives the smaller high-$p$ bound.
\end{remark}

\subsection{A High-\texorpdfstring{$p$}{p} Poisson Estimate}

For $z=(z_1,\ldots,z_n)\in\U^n$, let $P_n$ denote the product Poisson operator
\[
 P_n\psi(z)=
 \int_{\T^n}\prod_{k=1}^nP_{z_k}(\omega_k)\psi(\omega)\,dm_n(\omega),
\]
where
\[
 P_z(\omega)=\frac{1-|z|^2}{|\omega-z|^2},
 \qquad z\in\U,\quad \omega\in\T.
\]
The next lemma iterates the explicit one-variable $L^2\to L^{2m}$ estimate over the coordinates. It supplies the $L^2$ input for the high-$p$ theorem.

\begin{lemma}\label{lem:product-Poisson}
Let $m\in\N$, $m\geq2$, and let $\psi\in L^2(\T^n)$. Then
\begin{equation}\label{eq:product-Poisson}
 \int_{\U^n}|P_n\psi(z)|^{2m}\,d\mu_{\mathbf{m-2}}(z)
 \leq
 D_m^{2mn}
 \left(\int_{\T^n}|\psi(\omega)|^2\,dm_n(\omega)\right)^m.
\end{equation}
\end{lemma}

\begin{proof}
We first consider one variable. The case $p=2$ of Corollary \ref{cor:diagonal-Markovic} gives, for every $\varphi\in L^2(\T)$,
\begin{equation}\label{eq:one-variable-Poisson}
 \int_{\U}|P_1\varphi(z)|^{2m}\,d\mu_{m-2}(z)
 \leq
 D_m^{2m}
 \left(\int_{\T}|\varphi(\zeta)|^2\,dm_1(\zeta)\right)^m.
\end{equation}
Indeed, $P_1\varphi$ is harmonic in $\U$, and
\[
 \|P_1\varphi\|_{b^{2m}_{m-2}(\U)}
 \leq D_m\|P_1\varphi\|_{h^2(\U)}
 \leq D_m\|\varphi\|_{L^2(\T)}.
\]

We prove \eqref{eq:product-Poisson} first for continuous $\psi$ by induction on $n$. The case $n=1$ is \eqref{eq:one-variable-Poisson}. Suppose that the result has been proved in dimension $n-1$. Write
\[
 z=(z',z_n)\in\U^{n-1}\times\U,
 \qquad
 \omega=(\omega',\eta)\in\T^{n-1}\times\T.
\]
For each $\eta\in\T$, set
\[
 \psi_\eta(\omega')=\psi(\omega',\eta),
 \qquad
 V_\eta(z')=P_{n-1}\psi_\eta(z').
\]
For fixed $z'$, the function $z_n\mapsto P_n\psi(z',z_n)$ is the one-variable Poisson extension of the boundary function $\eta\mapsto V_\eta(z')$. Applying \eqref{eq:one-variable-Poisson} in the variable $z_n$ and then integrating over $z'\in\U^{n-1}$ gives
\begin{align*}
 \int_{\U^n}|P_n\psi(z)|^{2m}\,d\mu_{\mathbf{m-2}}(z)
 &\leq D_m^{2m}
 \int_{\U^{n-1}}
 \left(\int_{\T}|V_\eta(z')|^2\,dm_1(\eta)\right)^m
 d\mu_{\mathbf{m-2}}(z').
\end{align*}
By Minkowski's integral inequality,
\begin{align*}
 &\left[
 \int_{\U^{n-1}}
 \left(\int_{\T}|V_\eta(z')|^2\,dm_1(\eta)\right)^m
 d\mu_{\mathbf{m-2}}(z')
 \right]^{1/m} \\
 &\qquad\leq
 \int_{\T}
 \left(
 \int_{\U^{n-1}}|V_\eta(z')|^{2m}
 d\mu_{\mathbf{m-2}}(z')
 \right)^{1/m}dm_1(\eta).
\end{align*}
The induction hypothesis gives
\[
 \left(
 \int_{\U^{n-1}}|V_\eta(z')|^{2m}
 d\mu_{\mathbf{m-2}}(z')
 \right)^{1/m}
 \leq
 D_m^{2(n-1)}
 \int_{\T^{n-1}}|\psi(\omega',\eta)|^2\,dm_{n-1}(\omega').
\]
Combining these inequalities proves \eqref{eq:product-Poisson} for continuous $\psi$.

For general $\psi\in L^2(\T^n)$, choose $\psi_k\in C(\T^n)$ such that $\psi_k\to\psi$ in $L^2(\T^n)$. Applying the estimate we have already proved, to $\psi_k-\psi_\ell$, shows that $P_n\psi_k$ converges in $L^{2m}(\U^n,d\mu_{\mathbf{m-2}})$. For each fixed $z\in\U^n$, the product Poisson kernel belongs to $L^2(\T^n)$, and hence $P_n\psi_k(z)\to P_n\psi(z)$. An almost-everywhere convergent subsequence of the $L^{2m}$-convergent sequence therefore identifies its limit with $P_n\psi$. Passing to the limit in the estimate for $\psi_k$ proves \eqref{eq:product-Poisson} for every $\psi\in L^2(\T^n)$.
\end{proof}

\begin{proof}[Proof of Theorem \ref{thm:intro-high-p}]
Fix $0<\rho<1$ and put
\[
 f_\rho(z)=f(\rho z),\qquad z\in\U^n.
\]
The function $f_\rho$ is pluriharmonic in a neighbourhood of the closed polydisc, and hence it is separately harmonic and continuous on $\overline{\U^n}$. The iterated Poisson formula gives
\[
 f_\rho(z)=P_n[f_\rho|_{\T^n}](z).
\]
Since $p>2$, the function $w\mapsto |w|^{p/2}$ is convex on $\C\simeq\mathbb R^2$. The product Poisson kernel is positive and has total mass one, so Jensen's inequality gives
\[
 |f_\rho(z)|^{p/2}
 \leq
 P_n[|f_\rho|^{p/2}](z).
\]
Define
\[
 \psi_\rho(\omega)=|f(\rho\omega)|^{p/2},
 \qquad \omega\in\T^n.
\]
Then $|f(\rho z)|^{mp}\leq|P_n\psi_\rho(z)|^{2m}$. Lemma \ref{lem:product-Poisson} yields
\begin{align*}
 \int_{\U^n}|f(\rho z)|^{mp}\,d\mu_{\mathbf{m-2}}(z)
 &\leq
 \int_{\U^n}|P_n\psi_\rho(z)|^{2m}\,d\mu_{\mathbf{m-2}}(z)\\
 &\leq
 D_m^{2mn}
 \left(\int_{\T^n}|f(\rho\omega)|^p\,dm_n(\omega)\right)^m\\
 &\leq
 D_m^{2mn}\|f\|_{h^p(\U^n)}^{mp}.
\end{align*}
Letting $\rho\to1^-$ and using Fatou's lemma, we obtain
\[
 \int_{\U^n}|f(z)|^{mp}\,d\mu_{\mathbf{m-2}}(z)
 \leq
 D_m^{2mn}\|f\|_{h^p(\U^n)}^{mp}.
\]
Taking $mp$-th roots proves the theorem.
\end{proof}

\begin{remark}
The proof also works for $p=2$, since $w\mapsto|w|$ is convex. It then gives the constant $D_m^n$, which is weaker than the dimension-free constant $D_m$ when $n>1$. This is why Theorem \ref{thm:intro-high-p} is stated for $p>2$.

The same high-$p$ estimate follows by interpolation. Lemma \ref{lem:product-Poisson} gives
\[
 \|P_n\|_{L^2(\T^n)\to L^{2m}(\U^n,d\mu_{\mathbf{m-2}})}
 \leq D_m^n,
\]
whereas positivity of the Poisson kernel gives
\[
 \|P_n\|_{L^\infty(\T^n)\to L^\infty(\U^n)}\leq1.
\]
With $\theta=1-2/p$, the Riesz--Thorin theorem yields
\[
 \|P_n\|_{L^p(\T^n)\to L^{mp}(\U^n,d\mu_{\mathbf{m-2}})}
 \leq D_m^{n(1-\theta)}=D_m^{2n/p}.
\]
Thus interpolation and the direct Jensen argument give the same constant. The argument also applies, for $p\geq2$, to separately harmonic functions whose dilates are continuous on the closed polydisc and possess the iterated product Poisson representation used above.
\end{remark}

The constant function, the diagonal Riesz-transfer estimate, Theorem \ref{thm:intro-high-p}, and the preceding $p=2$ observation give
\[
 1\leq C_{p,m,n}\leq\Gamma_{p,m,n},
 \qquad p\geq2.
\]

The preceding bound immediately determines the large-$p$ behavior of the optimal diagonal constant.

\begin{corollary}
For fixed $m,n\in\N$ with $m\geq2$,
\[
 \lim_{p\to\infty}C_{p,m,n}=1.
\]
\end{corollary}

\begin{proof}
The constant function shows that $C_{p,m,n}\geq1$. On the other hand, Theorem \ref{thm:intro-high-p} gives $C_{p,m,n}\leq D_m^{2n/p}$ for $p>2$, and the latter quantity tends to $1$ as $p\to\infty$.
\end{proof}

Recall the constant $\Gamma_{p,m,n}$ defined in \eqref{eq:Gamma-intro}. The preceding diagonal estimates prove \eqref{eq:intro-diag} for every $p>1$. Applying this estimate to each factor and using H\"older's inequality completes the proof of the main polydisc theorem.

\subsection{Proof of Theorem \ref{thm:intro-main}}
\begin{proof}
H\"older's inequality and \eqref{eq:intro-diag} give
\begin{align*}
 \int_{\U^n}\prod_{j=1}^m|f_j|^{p_j}\,d\mu_{\mathbf{m-2}}
 &\leq
 \prod_{j=1}^m
 \left(\int_{\U^n}|f_j|^{mp_j}\,d\mu_{\mathbf{m-2}}\right)^{1/m}\\
 &\leq
 \prod_{j=1}^m\Gamma_{p_j,m,n}^{p_j}
 \prod_{j=1}^m\|f_j\|_{h^{p_j}(\U^n)}^{p_j}.
\end{align*}
This is the assertion of Theorem \ref{thm:intro-main}. In particular, if $p_1,\ldots,p_m>2$, then $\Gamma_{p_j,m,n}^{p_j}\leq D_m^{2n}$ for every $j$, and hence
\[
 \int_{\U^n}\prod_{j=1}^m|f_j(z)|^{p_j}\,d\mu_{\mathbf{m-2}}(z)
 \leq D_m^{2mn}
 \prod_{j=1}^m\|f_j\|_{h^{p_j}(\U^n)}^{p_j}.
\]
The proof is complete.
\end{proof}

\section{Pavlovi\'c--Dostani\'c-Type Inequalities}

We now turn to the ball--polydisc inclusion. We need the following vector-valued form of the Pavlovi\'c--Dostani\'c theorem for finite-dimensional complex Hilbert spaces. The coefficient argument is included for completeness.

\begin{lemma}\label{lem:PD-vector}
Let $E$ be a finite-dimensional complex Hilbert space and $H\in H^2(\U^n;E)$. Write $H^*$ for its radial boundary function. Then
\begin{equation}\label{eq:PD-vector}
 \int_{\partial\B_n}\|H(\zeta)\|_E^{2n}\,d\sigma_n(\zeta)
 \leq
 \left(
 \int_{\T^n}\|H^*(\omega)\|_E^2\,dm_n(\omega)
 \right)^n.
\end{equation}
Here the left-hand side is understood in the usual Hardy sense, namely as the increasing radial limit
\[
 \lim_{r\to1^-}
 \int_{\partial\B_n}\|H(r\zeta)\|_E^{2n}\,d\sigma_n(\zeta),
\]
which is defined (finite or not).
\end{lemma}

\begin{proof}
We modify the coefficient proof of Pavlovi\'c and Dostani\'c to the setting of Hilbert-space coefficients. We include the details because the notation is slightly heavier in the Hilbert-space-valued case.

First assume that $H$ is a polynomial, i.e.,
\[
 H(z)=\sum_\alpha c_\alpha z^\alpha,
\]
where the sum is over non-negative multi-indices $\alpha=(\alpha_1,\ldots,\alpha_n)$, each coefficient $c_\alpha$ belongs to $E$, and
\[
 |\alpha|=\alpha_1+\cdots+\alpha_n,
 \qquad \alpha!=\alpha_1!\cdots\alpha_n!,
 \qquad z^\alpha=z_1^{\alpha_1}\cdots z_n^{\alpha_n}.
\]
Let
\[
 G(z)=H(z)^{\otimes n}=H(z)\otimes\cdots\otimes H(z)
\]
be the $n$-fold tensor power of $H$. The tensor product is equipped with its natural Hilbert-space norm, so
\[
 \|G(z)\|^2=\|H(z)\|_E^{2n}.
\]
Write
\[
 G(z)=\sum_q d_qz^q.
\]
If $q=(q_1,\ldots,q_n)$, then the coefficient $d_q$ is obtained by collecting all products of coefficients whose multi-indices add up to $q$:
\[
 d_q=
 \sum_{\alpha^{(1)}+\cdots+\alpha^{(n)}=q}
 c_{\alpha^{(1)}}\otimes\cdots\otimes c_{\alpha^{(n)}}.
\]
Here the sum is over ordered $n$-tuples of multi-indices. By orthogonality of monomials on the unit sphere,
\[
 \int_{\partial\B_n}\|G(r\zeta)\|^2\,d\sigma_n(\zeta)
 =
 \sum_q \|d_q\|^2 r^{2|q|}
 \frac{(n-1)!\,q!}{(n-1+|q|)!},
\]
where $q!=q_1!\cdots q_n!$ and $|q|=q_1+\cdots+q_n$.

We now estimate $\|d_q\|^2$. Let $\mathcal D_q$ be the set of ordered decompositions
\[
 \alpha^{(1)}+\cdots+\alpha^{(n)}=q,
\]
and let $N_q=\#\mathcal D_q$. For each element of $\mathcal D_q$, set
\[
 v_{\alpha^{(1)},\ldots,\alpha^{(n)}}
 =c_{\alpha^{(1)}}\otimes\cdots\otimes c_{\alpha^{(n)}}.
\]
Then $d_q$ is the sum of these $N_q$ vectors. Hence the Hilbert-space Cauchy--Schwarz inequality gives
\[
 \|d_q\|^2
 \leq
 N_q
 \sum_{\alpha^{(1)}+\cdots+\alpha^{(n)}=q}
 \left\|c_{\alpha^{(1)}}\otimes\cdots\otimes c_{\alpha^{(n)}}\right\|^2.
\]
Since tensor-product norms multiply,
\[
 \left\|c_{\alpha^{(1)}}\otimes\cdots\otimes c_{\alpha^{(n)}}\right\|^2
 =
 \prod_{\ell=1}^n \|c_{\alpha^{(\ell)}}\|_E^2.
\]
Thus
\begin{equation}\label{eq:dq-estimate}
 \|d_q\|^2
 \leq
 N_q
 \sum_{\alpha^{(1)}+\cdots+\alpha^{(n)}=q}
 \prod_{\ell=1}^n \|c_{\alpha^{(\ell)}}\|_E^2.
\end{equation}
The number of ordered decompositions is
\[
 N_q=
 \prod_{k=1}^n \binom{n+q_k-1}{q_k}.
\]
Indeed, for each coordinate $k$, one has to distribute $q_k$ among $n$ non-negative integers, and the coordinates are independent.

We also need the elementary factorial estimate
\begin{equation}\label{eq:factorial-PD}
 N_q\frac{(n-1)!\,q!}{(n-1+|q|)!}\leq 1.
\end{equation}
For completeness, let us recall the verification. If $(n)_a=n(n+1)\cdots(n+a-1)$ denotes the rising factorial, then
\[
 \binom{n+q_k-1}{q_k}=\frac{(n)_{q_k}}{q_k!}.
\]
Therefore,
\[
 N_q\frac{(n-1)!\,q!}{(n-1+|q|)!}
 =
 \frac{\prod_{k=1}^n (n)_{q_k}}{(n)_{|q|}}.
\]
Since $(n)_a(n)_b\leq (n)_{a+b}$ for non-negative integers $a,b$, repeated use gives
\[
 \prod_{k=1}^n (n)_{q_k}\leq (n)_{q_1+\cdots+q_n}=(n)_{|q|},
\]
which proves \eqref{eq:factorial-PD}.

Combining Parseval's formula, \eqref{eq:dq-estimate}, and \eqref{eq:factorial-PD}, we obtain
\[
 \int_{\partial\B_n}\|H(r\zeta)\|_E^{2n}\,d\sigma_n(\zeta)
 \leq
 \sum_q r^{2|q|}
 \sum_{\alpha^{(1)}+\cdots+\alpha^{(n)}=q}
 \prod_{\ell=1}^n\|c_{\alpha^{(\ell)}}\|_E^2.
\]
The last double sum is exactly the expansion of an $n$-th power:
\[
 \sum_q r^{2|q|}
 \sum_{\alpha^{(1)}+\cdots+\alpha^{(n)}=q}
 \prod_{\ell=1}^n\|c_{\alpha^{(\ell)}}\|_E^2
 =
 \left(\sum_\alpha \|c_\alpha\|_E^2 r^{2|\alpha|}\right)^n.
\]
By orthogonality of monomials on the torus,
\[
 \sum_\alpha \|c_\alpha\|_E^2 r^{2|\alpha|}
 =
 \int_{\T^n}\|H(r\omega)\|_E^2\,dm_n(\omega).
\]
Hence, for polynomial $H$,
\[
 \int_{\partial\B_n}\|H(r\zeta)\|_E^{2n}\,d\sigma_n(\zeta)
 \leq
 \left(
 \int_{\T^n}\|H(r\omega)\|_E^2\,dm_n(\omega)
 \right)^n,
 \qquad 0<r<1.
\]

By continuity, the preceding polynomial inequality remains valid for $r=1$. For a general $H\in H^2(\U^n;E)$, fix $0<\rho<1$ and apply this radius-one estimate to the total-degree Taylor partial sums of $H(\rho\,\cdot)$. The latter function is holomorphic in the polydisc of radius $1/\rho>1$, so these partial sums converge uniformly on the closed unit polydisc, and hence also on the closed unit ball. Passing to the limit gives
\[
 \int_{\partial\B_n}\|H(\rho\zeta)\|_E^{2n}\,d\sigma_n(\zeta)
 \leq
 \left(
 \int_{\T^n}\|H(\rho\omega)\|_E^2\,dm_n(\omega)
 \right)^n.
\]
Finally, let $\rho\to1^-$. Since $z\mapsto\|H(z)\|_E^{2n}$ is plurisubharmonic, its spherical radial means over $\partial\B_n$ are increasing in the radius. Hence the left-hand side increases to the usual Hardy boundary integral on $\partial\B_n$, and the right-hand side tends to the $H^2(\U^n;E)$ boundary integral. This proves \eqref{eq:PD-vector}.
\end{proof}

The proof also uses the following unit-ball form of the reverse Riesz inequality. We state it at the exponent needed below; the surface measure is normalized as in Section~1. This is Chen and Hamada's theorem \cite[Theorem~2.1(B1(II))]{ChenHamada} specialized to $p=2n$, since $B_{2n}=D_n$.

\begin{lemma}\label{lem:CH-ball}
Let
\[
 F=H+\overline G
\]
be pluriharmonic in a neighbourhood of $\overline{\B_n}$, where $H$ and $G$ are holomorphic in that neighbourhood. If
\[
 \operatorname{Re}(G(0)H(0))\leq0,
\]
then
\[
 \left(
 \int_{\partial\B_n}|F(\zeta)|^{2n}\,d\sigma_n(\zeta)
 \right)^{1/(2n)}
 \leq D_n
 \left(
 \int_{\partial\B_n}
 \bigl(|H(\zeta)|^2+|G(\zeta)|^2\bigr)^n
 \,d\sigma_n(\zeta)
 \right)^{1/(2n)}.
\]
\end{lemma}

We are now ready to prove the ball--polydisc theorem.

\subsection{Proof of Theorem \ref{thm:intro-PD}}
\begin{proof}
Write
\[
 f=h+\overline g,\qquad g(0)=0,
\]
where $h$ and $g$ are holomorphic in $\U^n$. Define, as earlier,
\[
 H=(h,g),
\]
and
\[
 Q(z)=\|H(z)\|=\bigl(|h(z)|^2+|g(z)|^2\bigr)^{1/2}.
\]

Fix $0<r<1$ and apply Lemma \ref{lem:CH-ball} to
\[
 F(z)=f(rz),\qquad H_r(z)=h(rz),\qquad G_r(z)=g(rz).
\]
Since $g(0)=0$, we have
\[
 \operatorname{Re}(G_r(0)H_r(0))
 =\operatorname{Re}(g(0)h(0))=0\leq0,
\]
so the required sign condition is satisfied. Hence
\[
 \left(
 \int_{\partial\B_n}|f(r\zeta)|^{2n}\,d\sigma_n(\zeta)
 \right)^{1/(2n)}
 \leq
 D_n
 \left(
 \int_{\partial\B_n}Q(r\zeta)^{2n}\,d\sigma_n(\zeta)
 \right)^{1/(2n)}.
\]

Next, apply Lemma \ref{lem:PD-vector} to the vector-valued holomorphic function $z\mapsto H(rz)$. Since $Q=\|H\|$, we get
\[
 \int_{\partial\B_n}Q(r\zeta)^{2n}\,d\sigma_n(\zeta)
 \leq
 \left(
 \int_{\T^n}Q(r\omega)^2\,dm_n(\omega)
 \right)^n.
\]
Combining the last two estimates gives
\[
 \left(
 \int_{\partial\B_n}|f(r\zeta)|^{2n}\,d\sigma_n(\zeta)
 \right)^{1/(2n)}
 \leq
 D_n
 \left(
 \int_{\T^n}Q(r\omega)^2\,dm_n(\omega)
 \right)^{1/2}.
\]

It remains to identify the $L^2$-norm of $Q$ on the torus. For $0<s<1$,
\[
 |f(s\omega)|^2
 =|h(s\omega)+\overline{g(s\omega)}|^2
 =|h(s\omega)|^2+|g(s\omega)|^2
 +h(s\omega)g(s\omega)+\overline{h(s\omega)g(s\omega)}.
\]
Since $hg$ is holomorphic in $\U^n$,
\[
 \int_{\T^n}h(s\omega)g(s\omega)\,dm_n(\omega)=h(0)g(0)=0,
\]
and the conjugate term also has integral zero. Hence
\[
 \int_{\T^n}|f(s\omega)|^2\,dm_n(\omega)
 =
 \int_{\T^n}Q(s\omega)^2\,dm_n(\omega)
\]
for every $0<s<1$. In particular,
\[
 \left(
 \int_{\T^n}Q(r\omega)^2\,dm_n(\omega)
 \right)^{1/2}
 \leq
 \|f\|_{h^2(\U^n)}.
\]
Therefore
\[
 \left(
 \int_{\partial\B_n}|f(r\zeta)|^{2n}\,d\sigma_n(\zeta)
 \right)^{1/(2n)}
 \leq
 D_n\|f\|_{h^2(\U^n)}
\]
for every $0<r<1$. Taking the supremum over $r$ proves the theorem.
\end{proof}

\medskip
The surface estimate has a corresponding volume form. It follows by adding an auxiliary variable and using the marginal formula for normalized surface measure.

\begin{corollary}
Let $f\in h^2(\U^n)$. Then
\[
 \int_{\B_n}|f(z)|^{2n+2}\,d\nu_n(z)
 \leq
 D_{n+1}^{2n+2}\|f\|_{h^2(\U^n)}^{2n+2}.
\]
Equivalently,
\[
 \left(
 \int_{\B_n}|f(z)|^{2n+2}\,d\nu_n(z)
 \right)^{1/(2n+2)}
 \leq
 D_{n+1}\|f\|_{h^2(\U^n)}.
\]
\end{corollary}

\begin{proof}
For $0<\rho<1$, define a function in $\U^{n+1}$ by
\[
 F_\rho(z_1,\ldots,z_n,z_{n+1})=f(\rho z_1,\ldots,\rho z_n).
\]
The last variable is only auxiliary: $F_\rho$ does not depend on $z_{n+1}$. If $f=h+\overline g$, then
\[
 F_\rho=H_\rho+\overline{G_\rho},
\]
where
\[
 H_\rho(z_1,\ldots,z_{n+1})=h(\rho z_1,\ldots,\rho z_n),
 \qquad
 G_\rho(z_1,\ldots,z_{n+1})=g(\rho z_1,\ldots,\rho z_n).
\]
Thus $F_\rho$ is pluriharmonic in $\U^{n+1}$, and $G_\rho(0)=g(0)=0$.

For $0<s<1$,
\[
 \int_{\T^{n+1}}|F_\rho(s\omega)|^2\,dm_{n+1}(\omega)
 =
 \int_{\T^n}|f(\rho s\omega')|^2\,dm_n(\omega'),
\]
where $\omega'=(\omega_1,\ldots,\omega_n)$. Taking the supremum over $s$ gives
\[
 \|F_\rho\|_{h^2(\U^{n+1})}
 \leq
 \|f\|_{h^2(\U^n)}.
\]
Apply Theorem \ref{thm:intro-PD} in dimension $n+1$ to $F_\rho$. We get
\[
 \|F_\rho\|_{h^{2n+2}(\B_{n+1})}
 \leq
 D_{n+1}\|F_\rho\|_{h^2(\U^{n+1})}
 \leq
 D_{n+1}\|f\|_{h^2(\U^n)}.
\]
Because $\rho<1$, the function $F_\rho$ is defined in a neighbourhood of $\overline{\B_{n+1}}$. Hence
\[
 \int_{\partial\B_{n+1}}|F_\rho(\zeta)|^{2n+2}\,d\sigma_{n+1}(\zeta)
 \leq
 D_{n+1}^{2n+2}\|f\|_{h^2(\U^n)}^{2n+2}.
\]

The push-forward of $d\sigma_{n+1}$ under the projection onto the first $n$ coordinates is $d\nu_n$. Indeed, apart from a set of surface measure zero, we may write
\[
 \zeta=\bigl(z,\sqrt{1-|z|^2}\,e^{i\theta}\bigr),
 \qquad z\in\B_n,\quad 0\leq\theta<2\pi.
\]
Writing $dV$ for Lebesgue volume and $r(z)=\sqrt{1-|z|^2}$, the induced metric has determinant
\[
 r(z)^2\bigl(1+|\nabla r(z)|^2\bigr)=1.
\]
Thus the surface element in these coordinates is $dV(z)\,d\theta$; see also \cite[Section~1.4.5]{RudinBall}. Since the surface area of $\partial\B_{n+1}$ is $2\pi^{n+1}/n!$, normalization gives
\[
 d\nu_n(z)=\frac{n!}{\pi^n}\,dV(z)
\]
and, whenever $\Phi$ depends only on the first $n$ variables,
\[
 \int_{\partial\B_{n+1}}\Phi(\zeta_1,\ldots,\zeta_n)\,d\sigma_{n+1}(\zeta)
 =
 \int_{\B_n}\Phi(z)\,d\nu_n(z).
\]
Using this with $\Phi(z)=|f(\rho z)|^{2n+2}$ gives
\[
 \int_{\B_n}|f(\rho z)|^{2n+2}\,d\nu_n(z)
 \leq
 D_{n+1}^{2n+2}\|f\|_{h^2(\U^n)}^{2n+2}.
\]
Letting $\rho\to1^-$ and using Fatou's lemma gives
\[
 \int_{\B_n}|f(z)|^{2n+2}\,d\nu_n(z)
 \leq
 D_{n+1}^{2n+2}\|f\|_{h^2(\U^n)}^{2n+2},
\]
which completes the proof.
\end{proof}

\section{Endpoint and Sub-Endpoint Estimates}\label{sec:endpoint}

We begin with the failure of the ordinary strong endpoint estimate. For $0<a<1$, let
\[
 f_a(z)=\frac{1-a^2|z|^2}{|1-az|^2}
 =\operatorname{Re}\frac{1+az}{1-az},
 \qquad z\in\U.
\]
This positive harmonic function has circular means equal to $1$, and hence $\|f_a\|_{h^1(\U)}=1$. Put $\varepsilon=1-a$, write $z=re^{i\theta}$, and set $t=1-r$. For $a$ sufficiently close to $1$, on the region
\[
 2\varepsilon\leq t\leq\frac14,
 \qquad |\theta|\leq t,
\]
we have $1-ar=\varepsilon+t-\varepsilon t\asymp t$ and $1-r^2\asymp t$. Moreover,
\[
 f_a(re^{i\theta})
 =\frac{(1-ar)(1+ar)}{(1-ar)^2+2ar(1-\cos\theta)},
\]
so $1-\cos\theta\leq\theta^2/2$ gives $f_a(re^{i\theta})\geq c/t$. Therefore, for every $m\geq2$,
\[
 \int_{\U}|f_a(z)|^m\,d\mu_{m-2}(z)
 \geq c_m\int_{2\varepsilon}^{1/4}\frac{dt}{t}
 \longrightarrow\infty
 \qquad\text{as }a\to1^-.
\]
Thus no constant $C$ can make
\[
 \|f\|_{b^m_{m-2}(\U)}\leq C\|f\|_{h^1(\U)}
\]
valid for all harmonic functions. Considering $f_a(z_1)$ gives the same failure in every dimension. Nevertheless, the square-function method provides a natural substitute once the square function itself is controlled.

For a pluriharmonic function $f=h+\overline g$ with $g(0)=0$, put
\[
 H_f=(h,g),\qquad
 Q_f(z)=\bigl(|h(z)|^2+|g(z)|^2\bigr)^{1/2}.
\]
Define
\[
 h_Q^1(\U^n)=
 \{f=h+\overline g:\ h,g\in\Hol(\U^n),\ g(0)=0,\ H_f\in H^1(\U^n;\C^2)\},
\]
with
\[
 \|f\|_{h_Q^1(\U^n)}
 :=\|H_f\|_{H^1(\U^n;\C^2)}
 =\sup_{0<r<1}\int_{\T^n}Q_f(r\omega)\,dm_n(\omega).
\]
The pointwise inequality $|f|\leq\sqrt2Q_f$ gives
\[
 \|f\|_{h^1(\U^n)}\leq\sqrt2\|f\|_{h_Q^1(\U^n)}.
\]
Thus $h_Q^1(\U^n)$ is continuously embedded in $h^1(\U^n)$, and it provides a natural sufficient endpoint hypothesis for the square-function transfer. Under this hypothesis, the following estimate replaces the failed ordinary $h^1$ inclusion.

\begin{theorem}\label{thm:endpoint-square}
Let $m\in\N$, $m\geq2$. If $f\in h_Q^1(\U^n)$, then
\[
 \|f\|_{b^m_{\mathbf{m-2}}(\U^n)}
 \leq
 B_m\|f\|_{h_Q^1(\U^n)}.
\]
\end{theorem}

\begin{proof}
Since $H_f\in H^1(\U^n;\C^2)$, Corollary \ref{cor:square-Markovic}, applied with $p=1$, gives
\[
 \int_{\U^n}Q_f(z)^m\,d\mu_{\mathbf{m-2}}(z)
 \leq
 \|f\|_{h_Q^1(\U^n)}^m.
\]
Applying Lemma \ref{lem:weighted-B} with $q=m$ and $\boldsymbol\alpha=\mathbf{m-2}$, we obtain
\[
 \int_{\U^n}|f(z)|^m\,d\mu_{\mathbf{m-2}}(z)
 \leq
 B_m^m\int_{\U^n}Q_f(z)^m\,d\mu_{\mathbf{m-2}}(z).
\]
Taking $m$-th roots proves the theorem.
\end{proof}

H\"older's inequality turns the preceding endpoint inclusion into a multilinear estimate.

\begin{corollary}\label{cor:endpoint-multilinear}
Let $m\in\N$, $m\geq2$, and suppose that $f_j\in h_Q^1(\U^n)$, $j=1,\ldots,m$. Then
\[
 \int_{\U^n}\prod_{j=1}^m |f_j(z)|\,d\mu_{\mathbf{m-2}}(z)
 \leq
 B_m^m\prod_{j=1}^m \|f_j\|_{h_Q^1(\U^n)}.
\]
\end{corollary}

\begin{proof}
This follows immediately by applying H\"older's inequality and then Theorem \ref{thm:endpoint-square} to each factor.
\end{proof}

When $m=2$, we have $B_2=1$. Constant functions give equality in Theorem \ref{thm:endpoint-square} and in Corollary \ref{cor:endpoint-multilinear}; hence the constants in these two endpoint statements are sharp in this case.

For $0<p\le 1$, one can instead require the holomorphic and co-holomorphic parts to belong separately to holomorphic Hardy spaces. Set
\[
 h^p_{\mathrm{split}}(\U^n)
 =\{f=h+\overline g:\ h,g\in H^p(\U^n),\ g(0)=0\}.
\]
At $p=1$, this is the same set as $h_Q^1(\U^n)$. More precisely,
\[
 \max\{\|h\|_{H^1},\|g\|_{H^1}\}
 \leq\|f\|_{h_Q^1}
 \leq\|h\|_{H^1}+\|g\|_{H^1}
 \leq\sqrt2\|f\|_{h_Q^1}.
\]
For the last inequality, integrate $|h|+|g|\leq\sqrt2Q_f$ on the same internal tori $r\T^n$ and let $r$ tend to $1$; the radial means of $|h|$ and $|g|$ are increasing.
Thus the square-function and split conditions are equivalent at the endpoint. For $0<p<1$, the split condition is stronger than $f\in h^p(\U^n)$ in general. Under this hypothesis, the holomorphic inequality can be applied separately to the two parts.

\begin{proposition}
Let $m\in\N$, $m\geq2$, $0<p\leq1$, and set $q=mp$. Suppose
\[
 f=h+\overline g,\qquad h,g\in H^p(\U^n),\qquad g(0)=0.
\]
Then the following estimates hold.

\noindent{\rm (i)} If $q\geq1$, then
\[
 \|f\|_{b^q_{\mathbf{m-2}}(\U^n)}
 \leq
 \|h\|_{H^p(\U^n)}+\|g\|_{H^p(\U^n)}.
\]

\noindent{\rm (ii)} If $0<q<1$, then
\[
 \int_{\U^n}|f(z)|^q\,d\mu_{\mathbf{m-2}}(z)
 \leq
 \|h\|_{H^p(\U^n)}^q+\|g\|_{H^p(\U^n)}^q.
\]
Equivalently,
\[
 \|f\|_{b^q_{\mathbf{m-2}}(\U^n)}
 \leq
 \left(\|h\|_{H^p(\U^n)}^q+\|g\|_{H^p(\U^n)}^q\right)^{1/q}.
\]
\end{proposition}

\begin{proof}
For every holomorphic $F\in H^p(\U^n)$, Markovi\'c's diagonal inequality \eqref{eq:Markovic-polydisc} gives
\begin{equation}\label{eq:holomorphic-subendpoint}
 \int_{\U^n}|F(z)|^{mp}\,d\mu_{\mathbf{m-2}}(z)
 \leq
 \|F\|_{H^p(\U^n)}^{mp}.
\end{equation}
If $q=mp\geq1$, then Minkowski's inequality in $L^q(d\mu_{\mathbf{m-2}})$ and \eqref{eq:holomorphic-subendpoint} give
\[
 \|f\|_{b^q_{\mathbf{m-2}}}
 \leq
 \|h\|_{A^q_{\mathbf{m-2}}}+\|g\|_{A^q_{\mathbf{m-2}}}
 \leq
 \|h\|_{H^p}+\|g\|_{H^p}.
\]
If $0<q<1$, then the pointwise inequality $|a+b|^q\leq |a|^q+|b|^q$ and \eqref{eq:holomorphic-subendpoint} imply
\[
 \int_{\U^n}|f|^q\,d\mu_{\mathbf{m-2}}
 \leq
 \int_{\U^n}|h|^q\,d\mu_{\mathbf{m-2}}
 +\int_{\U^n}|g|^q\,d\mu_{\mathbf{m-2}}
 \leq
 \|h\|_{H^p}^q+\|g\|_{H^p}^q.
\]
This proves both assertions.
\end{proof}

\section{Explicit Lower Bounds and an Open Problem}

The constants in the main inequalities above are explicit, but they are not known to be sharp for the full pluriharmonic class. Recall that $C_{p,m,n}$ denotes the best constant in
\[
 \|f\|_{b^{mp}_{\mathbf{m-2}}(\U^n)}
 \leq C_{p,m,n}\|f\|_{h^p(\U^n)}.
\]
The constant function gives $C_{p,m,n}\geq1$, while the sharp constant for the holomorphic subclass is 1 by Markovi\'c's theorem. The optimal constants are nondecreasing with the dimension. Indeed, extending a function on $\U^n$ to $\U^{n+1}$ by making it independent of the last variable gives
\begin{equation}\label{eq:dimension-monotonicity}
 C_{p,m,1}\leq C_{p,m,2}\leq\cdots\leq C_{p,m,n}.
\end{equation}

The planar example of Kalaj and Me\v{s}trovi\'c \cite[Example~1.6]{KalajMestrovic} admits a weighted extension to every $m\geq2$. Together with the explicit dimension-free upper bound, it determines the limiting behavior of the optimal $p=2$ constants.

\begin{proposition}\label{prop:Lm-lower-bound}
Let $m\in\N$, $m\geq2$, and let
\[
 u_a(z)=\operatorname{Re}\frac{z}{1-az},
 \qquad 0<a<1.
\]
Then
\[
 \lim_{a\to1^-}
 \frac{\|u_a\|_{b^{2m}_{m-2}(\U)}}
 {\|u_a\|_{h^2(\U)}}=L_m,
\]
where $L_m$ is the explicit constant defined in \eqref{eq:L-definition}. Moreover,
\begin{equation}\label{eq:Lm-asymptotic}
 L_m=\sqrt2\left(1-\frac{\log 2}{4m}+o(m^{-1})\right),
 \qquad m\to\infty.
\end{equation}
\end{proposition}

\begin{proof}
The Fourier series of $z/(1-az)$ gives
\[
 \|u_a\|_{h^2(\U)}^2
 =\frac12\sum_{k=1}^{\infty}a^{2k-2}
 =\frac1{2(1-a^2)}.
\]
Put $\varepsilon=1-a$ and make the change of variables $z=1-\varepsilon w$, where $w=x+iy$. The image of $\U$ is
\[
 \Omega_\varepsilon
 =\{w:2x>\varepsilon|w|^2\}.
\]
Furthermore,
\[
 u_a(1-\varepsilon w)
 =\varepsilon^{-1}\operatorname{Re}
 \frac{1-\varepsilon w}{1+(1-\varepsilon)w},
 \qquad
 1-|1-\varepsilon w|^2
 =\varepsilon(2x-\varepsilon|w|^2).
\]
Let $dA$ denote planar Lebesgue area. Then
$d\mu_{m-2}(z)=\frac{m-1}{\pi}(1-|z|^2)^{m-2}\,dA(z)$,
while the change of variables above gives $dA(z)=\varepsilon^2\,dx\,dy$.
It follows that
\begin{align}\label{eq:scaled-ua}
 \varepsilon^m\|u_a\|_{b^{2m}_{m-2}(\U)}^{2m}
 &=\frac{m-1}{\pi}
 \int_{\Omega_\varepsilon}
 \left|\operatorname{Re}
 \frac{1-\varepsilon w}{1+(1-\varepsilon)w}\right|^{2m}
 \left(2x-\varepsilon|w|^2\right)^{m-2}\,dx\,dy.
\end{align}
For $0<\varepsilon\leq1/2$ and $w\in\Omega_\varepsilon$, we have
\[
 \frac{1-\varepsilon w}{1+(1-\varepsilon)w}
 =-\frac{\varepsilon}{1-\varepsilon}
 +\frac1{(1-\varepsilon)(\varepsilon+(1-\varepsilon)(1+w))}
.\]
Since $w\in\Omega_\varepsilon$ implies $x>0$, we have $\operatorname{Re}(1+w)=1+x>0$, and hence
\begin{align*}
 |\varepsilon+(1-\varepsilon)(1+w)|^2
 &=(1-\varepsilon)^2|1+w|^2
 +2\varepsilon(1-\varepsilon)\operatorname{Re}(1+w)+\varepsilon^2\\
 &\geq(1-\varepsilon)^2|1+w|^2
\end{align*}
Together with the preceding identity, this gives
\[
 \left|\frac{1-\varepsilon w}{1+(1-\varepsilon)w}\right|
 \leq\frac4{|1+w|}+2\varepsilon.
\]
Consequently, the integrand in \eqref{eq:scaled-ua} is bounded by a constant depending only on $m$ times
\[
 x^{m-2}|1+w|^{-2m}
 +\varepsilon^{2m}x^{m-2}\chi_{\Omega_\varepsilon}(w).
\]
Here $\chi_{\Omega_\varepsilon}$ denotes the characteristic function of $\Omega_\varepsilon$. Extend the integrand in \eqref{eq:scaled-ua} by zero to the half-plane $x>0$, and denote the extended integrand by $F_\varepsilon$. It converges pointwise to
\[
 F(x,y)=(2x)^{m-2}
 \left(\frac{1+x}{(1+x)^2+y^2}\right)^{2m}.
\]
Choose $c_m>0$ sufficiently large and set, on the half-plane $x>0$,
\[
 G(w)=c_mx^{m-2}|1+w|^{-2m},
 \qquad
 E_\varepsilon(w)=c_m\varepsilon^{2m}x^{m-2}
 \chi_{\Omega_\varepsilon}(w).
\]
The preceding estimate gives $F_\varepsilon\leq G+E_\varepsilon$, and the pointwise limit satisfies $F\leq G$. Moreover,
\[
 \int_{x>0}G(w)\,dx\,dy
 =c_m\int_0^\infty\int_{\mathbb R}
 \frac{x^{m-2}}{((1+x)^2+y^2)^m}\,dy\,dx<\infty.
\]
With $W=X+iY=\varepsilon w$, we also have
\[
 \|E_\varepsilon\|_{L^1}
 =c_m\varepsilon^m
 \int_{\{2X>|W|^2\}}X^{m-2}\,dX\,dY
 =O(\varepsilon^m),
\]
because the last domain is bounded. Therefore $\min\{F_\varepsilon,G\}\to F$ pointwise and is dominated by $G$, while
\[
 0\leq F_\varepsilon-\min\{F_\varepsilon,G\}
 \leq E_\varepsilon.
\]
Dominated convergence, together with the vanishing $L^1$-remainder, now gives
\begin{equation}\label{eq:scaled-limit}
 \lim_{a\to1^-}\varepsilon^m
 \|u_a\|_{b^{2m}_{m-2}(\U)}^{2m}
 =\frac{m-1}{\pi}2^{m-2}J_m,
\end{equation}
where
\[
 J_m=\int_0^\infty\int_{\mathbb R}
 x^{m-2}
 \left(\frac{1+x}{(1+x)^2+y^2}\right)^{2m}
 \,dy\,dx.
\]
The standard beta and gamma integral formulas yield
\begin{align*}
 J_m
 &=\sqrt\pi\frac{\Gamma(2m-\frac12)}{\Gamma(2m)}
 \int_0^\infty x^{m-2}(1+x)^{1-2m}\,dx\\
 &=\sqrt\pi\frac{\Gamma(2m-\frac12)}{\Gamma(2m)}
 \frac{\Gamma(m-1)\Gamma(m)}{\Gamma(2m-1)}.
\end{align*}
Since
\[
 \varepsilon\|u_a\|_{h^2(\U)}^2
 =\frac1{2(2-\varepsilon)}\longrightarrow\frac14,
\]
combining the preceding identity with \eqref{eq:scaled-limit} gives
\[
 \lim_{a\to1^-}
 \left(\frac{\|u_a\|_{b^{2m}_{m-2}(\U)}}
 {\|u_a\|_{h^2(\U)}}\right)^{2m}
 =4^m\frac{m-1}{\pi}2^{m-2}J_m
 =\frac{2^{3m-2}}{\sqrt\pi}
 \frac{\Gamma(2m-\frac12)\Gamma(m)^2}
 {\Gamma(2m)\Gamma(2m-1)}
 =L_m^{2m}.
\]
The last equality follows from the half-integer gamma identity and the definition \eqref{eq:L-definition}. Finally, Stirling's formula gives
\[
 L_m^{2m}\sim\frac{2^m}{\sqrt2},
\]
and hence
\[
 \log L_m
 =\frac12\log 2-\frac{\log 2}{4m}+o(m^{-1}).
\]
Taking exponents, we get \eqref{eq:Lm-asymptotic}.
\end{proof}

\begin{proof}[Proof of Theorem \ref{thm:intro-Lm-asymptotic}]
Proposition \ref{prop:Lm-lower-bound} gives $L_m\leq C_{2,m,1}$. Together with \eqref{eq:dimension-monotonicity} and Corollary \ref{cor:diagonal-Markovic} with $p=2$, this yields
\[
 L_m\leq C_{2,m,1}\leq C_{2,m,n}\leq D_m,
 \qquad n\in\N.
\]
The asymptotic formula for $L_m$ is \eqref{eq:Lm-asymptotic}. Since $D_m<\sqrt2$ and $L_m\to\sqrt2$, we also have
\[
 0\leq\sup_{n\in\N}|C_{2,m,n}-\sqrt2|
 \leq\sqrt2-L_m\longrightarrow0,
\]
which proves the uniform conclusion.
\end{proof}

For $m=2$, Proposition \ref{prop:Lm-lower-bound} recovers the example of Kalaj and Me\v{s}trovi\'c, since $L_2^4=5/2$. Hence, for every $n\in\N$,
\[
 L_2\leq C_{2,2,n}\leq D_2.
\]

The high-$p$ estimate has the expansion
\[
 D_m^{2n/p}
 =1+\frac{2n\log D_m}{p}+O_{m,n}(p^{-2}),
 \qquad p\to\infty.
\]
Here the implied constants depend only on $m$ and $n$. Thus $C_{p,m,n}=1+O_{m,n}(p^{-1})$. In contrast, $R_{p,m}\sim2p/\pi$, so the Riesz-transfer constant cannot be asymptotically sharp. When $m=2$, $n=1$, and $p>2$, the Poisson estimate gives
\[
 C_{p,2,1}\leq D_2^{2/p},
\]
which is precisely the upper bound obtained by Melentijevi\'c \cite[Section~9]{Melentijevic}.

Kalaj and Me\v{s}trovi\'c already noted that the sharpness of their planar $p=2$ upper estimate was open \cite[Remark~1.4]{KalajMestrovic}. To the best of our knowledge, even $C_{2,2,1}$ is not known exactly. We conclude with the main sharpness problem.

\begin{problem}
Let $1<p<\infty$, $m,n\in\N$, and $m\geq2$. Determine the exact constant $C_{p,m,n}$ and describe the extremal functions, if they exist. For $p>2$, determine when, if ever, $C_{p,m,n}=\Gamma_{p,m,n}$, and decide whether
\[
 \lim_{p\to\infty}\sup_{n\geq1}C_{p,m,n}=1
\]
holds for each fixed $m\geq2$.
\end{problem}

\section*{Statements and Declarations}
\noindent\textbf{Funding.}
The second author was partially supported by the Li Ka Shing Foundation STU--GTIIT Joint Research Grant (Grant No. 2024LKSFG06) and the Natural Science Foundation of Guangdong Province (Grant No. 2024A1515010467). The third author was supported by the National Natural Science Foundation of China (Grant No. 12271189), the Natural Science Foundation of Guangdong Province (Grant Nos. 2024A1515010467 and 2026A1515012333), the STU Scientific Research Initiation Grant (NTF25017T), the HQU Teaching Reform Project (HQJGKT2411), and the Fujian Alliance of Mathematics (Grant No. 2023SXLMMS07).

\noindent\textbf{Competing interests.}
The authors declare that they have no competing interests.

\noindent\textbf{Data availability.}
Data sharing is not applicable to this article, as no datasets were generated or analyzed during the current study.

\noindent\textbf{Author contributions.}
All authors contributed to the conception of the results, the proofs, and the preparation and revision of the manuscript. All authors read and approved the final manuscript.

\end{document}